\documentclass[10pt,twoside]{amsart}
\usepackage{amsmath, amsthm, amscd, amsfonts, amssymb, float, graphicx, color, xcolor, soul }
\usepackage[bookmarksnumbered, colorlinks, plainpages]{hyperref}
\usepackage{graphicx}
\usepackage{epstopdf}
\usepackage{rotating}
\sethlcolor{green}
\newcommand{\lcm}{\mathop{\mathrm{lcm}}}
\newtheorem{thm}{Theorem}[section]
\newtheorem{cor}[thm]{Corollary}
\newtheorem{lem}[thm]{Lemma}
\newtheorem{prop}[thm]{Proposition}
\newtheorem{defn}[thm]{Definition}

\newtheorem {thmS}{Theorem }

\newtheorem*{thmA}{Theorem A}
\newtheorem*{thmB}{Theorem B}
\newtheorem*{corC}{Corollary C}
\newtheorem*{corD}{Corollary D}
\newtheorem{proc}[thm]{}

\newtheorem*{rem}{Remark}

\numberwithin{equation}{section}
\def\pn{\par\noindent}

\begin{document}

\vspace{1.3 cm}
\title{On some graphs  associated with the finite alternating groups}
\author{D.~Bubboloni, Mohammad~A.~Iranmanesh and S.~M.~Shaker}

\thanks{{\scriptsize
\hskip -0.4 true cm MSC(2010): Primary: 05C25; Secondary: 20B30.
}}
\maketitle
\begin{abstract}  Let $P_0(A_n), \widetilde{P}_0(A_n), P_0(\mathcal{T}(A_n))$ and
$\mathcal{O}_0(A_n)$ be respectively the proper power graph, the proper quotient power graph,
the proper power type graph and  the proper order graph of the alternating group $A_n$,
for $n\geq 3.$ We determine the number of the components of  those graphs.
In particular, we prove that  the power graph $P(A_n)$ is $2$-connected if and only
if the power type graph $P(\mathcal{T}(A_n))$ is $2$-connected, if and only
if either $n = 3$ or none of  $n, n-1, n-2, \frac{n}{2}$ and ${\frac{n-1}{2}}$ is
a prime. We also give some information on the properties of those components.
\end{abstract}

\vskip 0.2 true cm
\pagestyle{myheadings}
\markboth{\rightline {\scriptsize  Bubboloni, Iranmanesh and Shaker}}
         {\leftline{\scriptsize  }}
\bigskip

\section{ Introduction }
\vskip 0.4 true cm

In the recent literature in group theory there are many examples of the usefulness of associating a graph $X(G)$ with every finite group $G$, requiring that isomorphic groups determine isomorphic corresponding graphs. Classic choices for $X(G)$ are: the prime graph, the prime graph for conjugacy class sizes or for complex irreducible characters degrees, the commuting graph, the non-commuting graph, the power graph and, of course, Cayley graphs. The process of passing from $G$ to $X(G)$
reduces the complexity, by focussing on some aspects of $G$ and allows the use of methods  from graph theory.
If the reduction of complexity is weak, little information about $G$ is lost and one can recover the group $G$, up to isomorphism, from the graph $X(G)$. This very desirable situation of course depends on the classes of groups under consideration. To be more precise, let $\mathcal{F}$ be the class of the finite groups and let $\mathcal{ U}$ and $ \mathcal{ V}$ be subclasses of $ \mathcal{ F}$ such that $\mathcal{ U}\subseteq \mathcal{ V}\subseteq\mathcal{F}$. Assume that, for every $G\in \mathcal{ V},$ a graph $X(G)$ is given.
Define the class of groups in $\mathcal{ U}$ that are $X$-{\it recognisable} in $\mathcal{ V}$ by
$$\mathcal{R}_{X}(\mathcal{U},\mathcal{V} )=\{G\in \mathcal{U} : \forall H\in\mathcal{V},\  X(H)\cong X(G) \Rightarrow H\cong G \}.$$
We can measure the level of adherence of $X$  to $ \mathcal{U}$, with respect to $\mathcal{V}$,  by the broadness of $\mathcal{R}_{X}(\mathcal{U},\mathcal{V} )$ with respect to $\mathcal{U}$. 
The best level of adherence is realized  for those choices of $X$ such that $\mathcal{R}_{X}(\mathcal{U},\mathcal{V} )=\mathcal{U}$; in that case we say that $X$ recognises $\mathcal{U}$ in $\mathcal{V}$. Frequent applications of that concept in the literature are given when $\mathcal{V}=\mathcal{U}$ or $\mathcal{V}=\mathcal{F}$. If $\mathcal{R}_{X}(\mathcal{U},\mathcal{U} )=\mathcal{U}$
 we say that $X$ recognizes $\mathcal{U}$; if $\mathcal{R}_{X}(\mathcal{U},\mathcal{F} )=\mathcal{U}$ we say that $X$ recognizes $\mathcal{U}$ among the finite groups.
Even though the adherence level is not the best possible, the definition of $X$ is considered to be interesting if $\mathcal{R}_{X}(\mathcal{U},\mathcal{V} )$ contains known classes of  groups, such as nilpotent groups, solvable groups, abelian groups, simple groups and so on.

In this paper, all the considered classes of groups  are included in $\mathcal{F}$ and $G$ always denotes a finite group;  every graph $X=(V_X,E_X)$  is finite, undirected, simple and reflexive. Having a loop on every vertex simplify many arguments in the general theory developed in \cite{bub}
to recover the number and the structure of the components of a graph by the knowledge of the components of its quotient graphs.

We are mainly focused on the \emph{power graph} $P(G),$ that is, the graph
with  vertex set $G$ and having an edge between two vertices if one of them is the power of the other. 
Many interesting results  on power graphs are collected in the survey \cite{sur}, where a wide list of open problems and references is given. Nowadays, the study of the
power graphs is  far to be completed and the recent literature  shows a growing interest about the graph theoretical properties of $P(G)$. For instance, in
\cite{wang2} it is shown that the power graph has a transitive orientation and  a closed formula for the metric dimension of $P(G)$ is established. In \cite{ wang1},
relying on a fundamental result about  groups with isomorphic
power graphs (\cite[Proposition 1]{cam2}), the full automorphism group of $P(G)$ is described. In many cases
such as \cite{ df, pya,shi,df2,Ma}, those properties are investigated in relation to the group theoretical properties of $G$. In particular,  the groups $G$ for which  $P(G)$ is planar or toroidal or projective are
classified in \cite{df2}; in \cite{Ma} it is presented a characterization of the chromatic number of $P(G)$
and  the groups whose power graphs are uniquely colorable, split or unicyclic are classified.
In our opinion, a motivation for the studying of the power
graph is given by its realizing a high level of adherence to the class of finite groups. Indeed,
by \cite[Theorem 1]{shi} and \cite[Theorem 15]{filo}, we have:

\begin{thmS}\label{first}Let $\mathcal{U}$  be the union of the following classes of  groups:
\begin{itemize}
\item[(1)] the class $\mathcal{S}$ of simple groups ;
\item[(2)] the symmetric groups;
\item[(3)] the automorphism groups of a sporadic simple group;
\item[(4)] the class $\mathcal{C}$ of cyclic groups;
\item[(5)] the dihedral groups;
\item[(6)] the generalized quaternion groups.
\end{itemize}
Then $\mathcal{R}_{P}(\mathcal{F},\mathcal{F} )\supseteq \mathcal{U}$.

\end{thmS}

 Theorem \ref{first} says that the power graph recognizes $\mathcal{U}$ among the finite groups. In particular, as a very remarkable fact, the power graph recognizes $\mathcal{S}$ among the finite groups. On the other hand, $\mathcal{R}_{P}(\mathcal{F},\mathcal{F} )\neq \mathcal{F}$. Indeed it is easily seen that  $\mathcal{R}_{P}(\mathcal{F},\mathcal{F} )$ has empty intersection with the class of $p$-groups $G$, with $|G|\geq p^3$ and $\mathrm{Exp}(G)=p$. First of all  note that any $p$-group $G$, with $\mathrm{Exp}(G)=p$ admits a partition given by its distinct subgroups.
Thus $P(G)$ is formed by $\frac{|G|-1}{p-1}$ complete graphs on $p-1$ vertices and by the edges $\{x,1\}$, for $x\in G.$ In particular, the structure of $P(G)$ depends only on the size of $G.$ Now assume that a certain $p$-group  $G$ of exponent $p$, with $|G|\geq p^3$, is recognizable among the finite groups. Pick a
 $p$-group  $H$ of exponent $p$ with $|G|=|H|$, chosen such that only one among $G$ and $H$ is abelian. Then $P(G)\cong P(H)$ while $G\not \cong H,$ a contradiction.

The exact description of $\mathcal{R}_{P}(\mathcal{F},\mathcal{F} )$ is, at the best of our knowledge, still wide open.  In particular, does $\mathcal{R}_{P}(\mathcal{F},\mathcal{F} )$ contains  the full class of  the almost simple groups?

The previous example also shows that if $\mathcal{A}$ is the class of the abelian groups, then $\mathcal{R}_{P}(\mathcal{A},\mathcal{F} )\subsetneq \mathcal{A}$. Note instead that, by \cite[Theorem 1]{cam1}, the power graph recognizes the abelian groups, that is $\mathcal{R}_{P}(\mathcal{A},\mathcal{A} )=\mathcal{A}.$ It follows that $\mathcal{R}_{P}(\mathcal{A},\mathcal{F} )$ contains the $p$-groups of order $p^2$ and so, using Theorem  \ref{first},  $\mathcal{C}\subsetneq\mathcal{R}_{P}(\mathcal{A},\mathcal{F} )\subsetneq \mathcal{A}$.  The exact determination of the class $\mathcal{R}_{P}(\mathcal{A},\mathcal{F} )$ is a further open problem.

The fact that the power graph  recognizes $\mathcal{S}$ among the finite groups, indicates that it is possible to shed light on simple groups through the power graph. We know that if $G\in \mathcal{S}$ and $H\in \mathcal{F}$ are such that $P(H)\cong P(G)$, then $H\cong G.$ To reach an effective control of being $P(H)\cong P(G)$, when $G\in \mathcal{S}$ is given, we need a deep knowledge of the graph theoretical properties of $P(G).$
Among those, connectivity is one of the most natural. Since $P(G)$ is always connected, the focus is on $2$-connectivity. Recall that a graph $X$ is called $2$-connected if, for every $x\in V_{X}$, the $x$-deleted subgraph $X- x$  is connected. It is immediately checked that $P(G)$ is $2$-connected if and only if the  \emph{proper power graph} $P_0 (G)$, defined as the $1$-deleted subgraph of $P (G),$  is connected. The study of $2$-connectivity of $P(G)$ for $G\in \mathcal{S}$, begun with a question posed by Pourgholi, Yousefi-Azari and Ashrafi \cite[Question 13]{pya} about the existence of finite simple groups with $2$-connected power graphs. In \cite{aa} some simple groups are analyzed and for all of them it is shown that the power graph is not $2$-connected. We believe that it could be interesting to find  the number $c_0(G)$ of components of $P_0 (G)$ for $G\in \mathcal{S}$, the aim being a grouping of simple groups having the same  $c_0(G)$.
This research echoes the counting of the components of the prime graph for simple groups in \cite{wi}, and could lead to an analogous new understanding of the simple groups.

In this paper, we compute $c_0(A_n)$ for all $n\geq 3,$ and we found among other things, an infinite family of alternating groups with $2$-connected power graph. We also add a description of some graph theoretical properties of the components.

\section{Notation and main results}
\bigskip
In \cite{bis} we studied the $2$-connectivity of  $P(G)$, for  $G\in \mathcal{F}$, as an application of the general theory developed in \cite{bub}.  Theorem B in \cite{bub} gives an explicit formula  for counting the components of a graph $X$ by the knowledge of the components of its quotient graph $Y$ when the projection is an  orbit homomorphism. We briefly  recall the main concepts and definitions needed for the sequel. For further details, the reader is referred to \cite{bub}.

\subsection{Graphs and homomorphisms}
For a finite set $A$ and $k\in \mathbb{N}$, let $\binom{A}{k}$ be the set of the subsets of $A$ of size $k.$
In this paper, as in  \cite{bub} and \cite{bis}, a graph $X=(V_X,E_X)$ is a pair of finite sets such that $V_X\neq\varnothing$ is the set of vertices  and $E_X$ is the  set  of edges given by the union of the set of  loops $L_X=\binom{V}{1}$ and of a set  $E^*_X\subseteq \binom{V}{2}$. Note that while there is a loop on every vertex, the set $E^*_X$ may be empty. If  $e\in E^*_X$  we say that $e$ is a {\it proper edge}.
We usually specify the edges of a graph giving only the edges in $E^*_X$.  $X$ is called complete if $E_X=\binom{V}{1}\cup \binom{V}{2}.$

From now on, let  $X=(V_X,E_X)$ and $Y=(V_Y,E_Y)$ be  two  graphs and let $\varphi:V_X \rightarrow V_Y$ be  a map. For every $y\in V_Y,$ the subset $\varphi^{-1}(y)$ of $V_X$  is called the {\it fibre} of $\varphi$ on $y$. The relation $\sim_{\varphi}$
 on $V_X$ defined, for every $x,y\in V_X$, by $x\sim_{\varphi}y$ if $\varphi(x)=\varphi(y),$ is an equivalence relation. The equivalence classes of  $\sim_{\varphi}$ are called $\varphi$-{\it cells} and coincide with the nonempty fibres of $\varphi$. We call $\sim_{\varphi}$ the equivalence relation induced by $\varphi$ and denote the corresponding quotient graph  by $X/\hspace{-1mm}\sim_{\varphi}.$
$\varphi$ is called a homomorphism from $X$ to $Y$  if, for every $x_1,x_2\in V_X,$ $\{x_1,x_2\}\in E_X$ implies $\{\varphi(x_1),\varphi(x_2)\}\in E_Y$.  In that case we use the notation $\varphi:X \rightarrow Y$. The set of homomorphisms from $X$ to $Y$ is denoted by $\mathrm{Hom}(X,Y)$.
Let $\varphi\in \mathrm{Hom}(X,Y)$. Then $\varphi$ naturally induces a map from $E_X$ to $E_Y$, denoted again by $\varphi$, associating to every $e=\{x_1,x_2\}\in E_X$ the edge $\varphi(e)=\{\varphi(x_1),\varphi(x_2)\}\in E_Y$. $\varphi$ is called a $2$-{\it homomorphism} if, for every $e\in E^*_X$, $\varphi(e)\in E^*_Y$;
 surjective (injective) if $\varphi:V_X\rightarrow V_Y$ is surjective (injective). If $\hat{X}$ is a subgraph of $X$, we define the image of $\hat{X}$ by $\varphi$ as the subgraph of $Y$ given by $\varphi(\hat{X})=(\varphi(V_{\hat{X}}), \varphi(E_{\hat{X}})).$ $\varphi$ is called
 {\it complete} if $\varphi(X)=Y$. Note that being complete is a stronger than being  surjective.
 We call $\varphi$ {\it locally strong} if for every $x_1,x_2\in V_X$, $\{\varphi(x_1),\varphi(x_2)\}\in E_Y$ implies the existence of $\tilde{x}_2\in \varphi^{-1}(\varphi(x_2))$  such that $\{x_1,\tilde{x}_2\}\in E_X;$  {\it pseudo-covering}  if it is surjective and locally strong;  an  {\it orbit homomorphism} if there exists $\mathfrak{G}\leq  \mathrm{Aut}(X)$ such that the partition of $V_X$ into $\varphi$-cells coincides with the partition of $V_X$ into $\mathfrak{G}$-orbits.
If $\varphi$ is an orbit homomorphism with respect to $\mathfrak{G}$ we shortly say that $\varphi$ is $\mathfrak{G}$-{\it consistent} or that $\mathfrak{G}$ is $\varphi$-{\it consistent}. Finally $\varphi$ is called {\it tame}  if  every fibre of $\varphi$ is connected.
\subsection{Components, admissibility, vertex deleted subgraphs}\label{components-intro} A component of $X$ is a maximal connected subgraph of $X$. The component of $X$ containing the vertex $x\in V_X$ is denoted by $C_{X}(x)$ or more simply, when no confusion is possible, by $C(x).$
The set of components of $X$ is denoted by $ \mathcal{C}(X)$.

Given $U\subseteq V_X$ and $y\in V_Y$, define the {\it multiplicity} of $y$ in $U$, through the map $\varphi$, by the non-negative integer $k_{U}(y)=|U\cap \varphi^{-1}(y) |.$
We say that $y$ is  {\it admissible} for $U$ (or $U$ is admissible for $y$), if $k_{U}(y)>0$. The set $\varphi(U)$ is therefore the set of vertices of $Y$ admissible for $U.$
If $\hat{X}$ is a subgraph of $X$ we adopt the same language referring to its set of vertices $V_{\hat{X}}$. In particular  $k_{\hat{X}}(y)$ is defined by  $k_{V_{\hat{X}}}(y)$.
 Note that $k_{X}(y)$ is simply the size of the fibre  $\varphi^{-1}(y)$. The subgraphs on which we are focused are the components of $X$.
The set of components of $X$ admissible for $y$  is denoted by $\mathcal{C}(X)_{y}=\{C\in\mathcal{C}(X): k_{C}(y)> 0\}.$

A homomorphism $\varphi\in \mathrm{Hom}(X,Y)$ is called {\it component equitable} if for every $y\in V_Y$ and every $C,\overline{C}\in \mathcal{C}(X)_{y}$, $k_{C}(y)=k_{\overline{C}}(y)$.
In \cite[Propositions 5.9 and 6.9]{bub} it is shown that every orbit complete homomorphism is component equitable and pseudo-covering.

Let  $x_0\in V_X$ be fixed. The vertex deleted subgraph is the graph $X_0=((V_X)_0,(E_X)_0)$ having vertex set $(V_X)_0=V_X\setminus\{x_0\}$ and edge set $(E_X)_0$ given by the edges in $E_X$ not incident to $x_0$. We use that uniform notation disregarding, from the notational point of view, the particular $x_0$ used.
 Moreover  we write $\mathcal{C}_0(X)$ for the set of the components of $X_0$ as well as $c_0(X)$  for their number. In order to improve readability, we occasionally  introduce some slight variation of that notational rule.

\subsection{The  proper power graph and its quotients}
The {\it power graph} of $G$ is the graph $P (G)$ with $V_{P (G)}=G$ and edge set $E$ where, for $x,y\in G$, $\{x,y\}\in E$ if there exists $m\in\mathbb{N}$ such that $x=y^m$ or $y=x^m$. Let $G_0=G\setminus\{1\}.$ The {\it proper  power graph} $P_0 (G)=(G_0,E_0)$ is defined as the $1$-deleted subgraph  of $P(G).$ The number of components  of $P_0 (G)$ is denoted by  $c_0(G)$.

 In \cite{bis} we  got a formula for  $c_0(G)$ through the consideration of a series of quotients of $X=P_0 (G)$.  For a particular class of permutation  groups, the so-called fusion controlled permutation groups (\cite[Definition 6.1]{bis}), that formula became particularly concrete (\cite[Theorem A]{bis}) by considering a relevant quotient graph, the power type graph. Throughout the paper  let $n$ indicate a natural number. The symmetric group $S_n$ and the alternating group $A_n$ are interpreted as naturally acting on the set $N=\{1,\dots,n\}$ and with identity element $id$. They are both  fusion controlled and so they share a similar method for the interpretation and counting of their proper power graph components. Recall that $G\leq S_n$ is called  fusion controlled if  for every $ \psi\in G$ and  $x\in S_n$ such that $\psi^x\in G$, there exists $y\in N_{S_n}(G)$ such that $\psi^x=\psi^y.$ We present an overview of the quotient graphs, introduced in \cite{bis},  that we are going to use.

The \emph{quotient power graph} $\widetilde{P}(G)$ is the quotient graph of
$P(G)$ with respect to the equivalence relation identifying $x,y\in G$ if
$\langle x\rangle=\langle y\rangle$. The vertex set  of $\widetilde{P}(G)$, denoted by
$[G]$, is formed by the corresponding equivalence classes $[x]$, for $ x\in G$, and the edge set is denoted  by $[E].$ We define the order of $[x]$ as the order of $x.$ For every $X\subseteq G$, we also set $[X]=\{[x]\in [G]:x\in X\}$.
The  {\it proper quotient power graph}
$\widetilde{P}_0(G)=([G]_0, [E]_0)$ is defined as the $[1]$-cut subgraph of $\widetilde{P}(G)$. Note that we have set
$[G]_0=[G]\setminus\{[1]\}$. We recall a very useful link between edges in the proper power graph and in the proper quotient power graph (\cite[Lemma 3.5]{bis}).
\begin{equation}\label{lato}
\forall x,y\in G_0,\  \{[x], [y]\}\in [E]_0 \mbox{ if and only if } \{x,y\}\in E_0.
  \end{equation}
The projection $\pi:G_0\rightarrow [G]_0$, defined by $\pi(x)=[x]$ for all $x\in G_0$ gives a pseudo-covering  homomorphism  between the graph $P_0(G)$ and its tame quotient $\widetilde{P}_0(G)$ (\cite[Definition 3.2 and Lemma 3.6]{bis}). As a consequence, the number $\widetilde{c}_0(G)$ of the components of
$\widetilde{P}_0(G)$  is equal to $c_0(G)$.

The {\it order graph}  $\mathcal{O}(G)$ is the graph
with vertex set $O(G)=\{ o(g): g\in G\}$, where for $a,b\in O(G)$, $\{a,b\}\in E_{\mathcal{ O}(G)}$
if $a$ divides $b$ or $b$ divides $a$.
The {\it proper order graph}
$\mathcal {O}_0(G)$ is the $1$-deleted subgraph of $\mathcal{O}(G) $.
In  \cite[Corollary 4.3]{bis}, we proved that $\mathcal {O}_0(G)$ is a quotient of $\widetilde{P}_0(G)$
 which usually simplifies in a too drastic way
the complexity of $\widetilde{P}_0(G)$. In fact,  the natural map
$\widetilde{o}:[G]_0\rightarrow O_0(G)$ defined by
$\widetilde{o}([\psi])=o(\psi)$ for all $[\psi]\in [G]_0$, induces a complete  $2$-homomorphism $\widetilde{o}$ from $\widetilde{P}_0(G)$ onto $\mathcal {O}_0(G)$ (\cite[Proposition 4.2]{bis}) which is not, in general, pseudo-covering (\cite[Example 4.4]{bis}). We denote by $c_0(\mathcal{O}(G))$ the number of components of $\mathcal {O}_0(G)$.
Note that $\{a,b\}\in E^*_{\mathcal{ O}_0(G)}$ if and only if $a$ divides $b$ or $b$ divides $a$, $a\neq b$ and both $a$ and $b$ are different from $1.$

Finally we recall the definition of the  power type graph $P(\mathcal{T}(G))$,  for $G\leq S_n$  (\cite[Section 5.4 ]{bis}.
We start recalling the definition of power in the set $\mathcal{T}(n)$ of the partitions of $n$. Recall that a partition of $n$ is an unordered $r$-tuple
$T=[x_1,\dots,x_r]$, with $x_i\in\mathbb{N}$ for all $i\in
\{1,\dots,r\}$ and $r\in \mathbb{N}$, such that $n=\sum_{i=1}^{r}x_i.$
The $x_i$ are  the terms of $T$.
Given $T\in \mathcal{T}(n)$, let $m_1<\dots < m_k$ be its $k$ distinct terms, for some $k\in \mathbb{N}$; if $m_j$ appears $t_j\geq 1$ times in $T$ we use the notation $T=[m_1^{t_1},..., m_k^{t_k}]$ and say that $t_j$ is the  multiplicity of $m_j$.  We will accept, in some occasions,  the multiplicity $t_j=0$ simply  to say that a certain natural number $m_j$ does not appear as a term in $T.$  We usually omit to write down the multiplicities equal to $ 1$. The partition $[1^n]$ is called the  trivial partition and we put $\mathcal{T}_0(n)=\mathcal{T}(n)\setminus \{[1^n]\}$.
We define $\lcm(T)=\lcm\{m_i\}_{i=1}^k$ and $\gcd (T)=\gcd\{m_i\}_{i=1}^k$; the  order of $T$  is defined by $\lcm(T)$ and written as $o(T).$ For $T=[x_1,..., x_r]\in\mathcal{T}(n)$ and  $a\in \mathbb{N}$, the power of $T$ of exponent $a$ is defined as the partition $T^a$ having as terms $\frac{x_i}{\gcd(a,x_i)}$  repeated $\gcd(a,x_i)$ times for  all $i\in \{1,\dots,r\}$. We say that $T^a$ is a  proper power of $T$ if $[1^n]\neq T^a\neq T.$ Note that $T^a$ is a proper power of $T$ if and only if $\gcd(a, o(T))\neq 1,o(T).$ For instance, $[2,5,6]^2=[1^2,3^2,5]$ is a proper power of $[2,5,6]\in \mathcal{T}(13).$

Let $\psi\in S_n.$ The  type of  $\psi$ is the partition of $n$ given by  the unordered list
$T_{\psi}=[x_1,...,x_r]$ of the sizes $x_i$  of the $r$ orbits  of $\psi$ on $N.$ In particular, $T_{id}=[1^n].$  The type map $t:S_n\rightarrow \mathcal{T}(n)$ which maps $\psi$ into $T_{\psi}$ is surjective
and it is well known that $T_{\psi}=T_{\varphi}$ if and only if $\psi, \varphi\in S_n$ are conjugate in $S_n$.  If $X\subseteq S_n$, then $t(X)$ is  the set of types admissible for $X$ in the sense of  \cite[Section 4.1]{bub}, and is  denoted by $\mathcal{T}(X)$.
 We also set  $\mathcal{T}_0(X)=\mathcal{T}(X)\setminus \{ [1^n]\}$. Note that $\mathcal{T}(S_n)=\mathcal{T}(n)\supsetneq \mathcal{T}(A_n)$ for all $n\geq 2.$

Given $G\leq S_n$,  the {\it power type graph}  $P(\mathcal{T}(G))$ is defined as the graph with vertex set $\mathcal{T}(G)$ and edge set
 $E_{\mathcal{T}(G)}$, where $\{T_1, T_2\}\in E_{\mathcal{T}(G)}$,  for some $T_1, T_2\in \mathcal{T}(G)$,  if there exists $a\in\mathbb{N}$ such that $T_1=T_2^a$ or $T_2=T_1^a.$
The  {\it proper power type graph}  $P_0(\mathcal{T}(G))=(\mathcal{T}_0(G), E_{\mathcal{T}_0(G)})$ is defined as the $[1^n]$-deleted subgraph of  $P(\mathcal{T}(G))$. Note that given $\psi,\varphi\in G_0$, there is a proper edge incident to $T_{\psi}, T_{\varphi}\in \mathcal{T}_0(G)$ in $P_0(\mathcal{T}(G))$  if and only if one of $T_{\psi}, T_{\varphi}$  is a proper power of the other. We denote by $c_0(\mathcal{T}(G))$ the number of components of $P_0(\mathcal{T}(G))$.
Recall that $\mathcal {O}_0(G)$ may be seen as a quotient of $P_0(\mathcal{T}(G))$ (\cite[Proposition 5.6]{bis}).

For every $\psi\in S_n$, the  type of  $[\psi]\in [S_n]$ is defined  by  $T_{[\psi]}=T_{\psi}$. The  map
$\widetilde{t}: [G]_0\rightarrow \mathcal{T}_0(G)$ defined by $\widetilde{t}([\psi])=T_{\psi}$ for all $[\psi]\in [G]_0,$
gives a complete $2$-homomorphism from $P_0(\mathcal{T}(G))$  to  $\widetilde{P}_0(G)$ so that $P_0(\mathcal{T}(G))$ is a quotient of $\widetilde{P}_0(G)$ (\cite[Proposition 5.4]{bis}).  For $X\subseteq S_n$  let  $[X]=\{[x]\in[S_n] :  x\in X\}$. Then, accordingly to \cite[Section 4.1]{bub}, $\widetilde{t}([X])=t(X)=\mathcal{T}(X)$ is  the set of   types admissible for $[X]$.  If $\hat{X}$ is a subgraph of $ \widetilde{P}_0(G)$ the set of types admissible for $\hat{X}$, denoted  by $ \mathcal{T}(\hat{X})$,  is given by the set of types admissible for $V_{\tilde{X}}.$ In particular, for $C$ component of $ \widetilde{P}_0(A_n)$ we have $ \mathcal{T}(C)=\{T\in \mathcal{T}(G_0):\hbox{there exists}\  [\psi]\in V_C\  \hbox{with}\ T_{\psi}=T\}.$

In the fusion controlled case,  $\widetilde{t} $ is also an orbit homomorphism (\cite[Proposition 6.2]{bis}) and thus we can deduce the number
of components of $\widetilde{P}_0(G)$ from those of  $P_0(\mathcal{T}(G))$ (\cite[Theorem A]{bis}).

In \cite{bis}, we proved that the graphs $P_0(G),\widetilde{P}_0(G),$ $ P_0(\mathcal{T}(G)),$ $ \mathcal{O}_0(G)$ are very strictly related  and  we computed simultaneously the number of components of those graphs for $G=S_n$ (\cite[Theorem B]{bis}).  In particular we saw
that, for $n\geq 8$, the components of $\widetilde{P}_0(S_n), P_0(\mathcal{T}(S_n))$ and
$\mathcal{O}_0(S_n)$, apart from one, are isolated vertices and that the $2$-connectivity
of $P(S_n)$ is equivalent to that of $P(\mathcal{T}(S_n))$ as well as to that of
$\mathcal{O}(S_n)$, that is, to being $n = 2$ or none of $n$ and $ n-1$ a prime (\cite[Corollary C]{bis}).
Here, we find analogous results about $A_n$ applying the same general algorithmic procedure.
Since, for $n\in\{1,2\}$, $P_0(A_n)$  is the empty graph, we deal with the set of graphs
$$\mathcal {G}=\{P(A_n), \widetilde{P}(A_n), P(\mathcal{T}(A_n)), \mathcal{O}(A_n)\}$$
and the set of their corresponding proper graphs
\begin{equation}\label{G0}
\mathcal {G}_0=\{P_0(A_n),\widetilde{P}_0(A_n), P_0(\mathcal{T}(A_n)), \mathcal{O}_0(A_n)\},
\end{equation}
for  $n\geq 3$. If $X\in \mathcal {G}$, we denote the corresponding proper graph in
$\mathcal {G}_0$ with $X_0$. Note that every $X\in \mathcal {G}$ is connected because it contains a vertex adjacent to any other vertex and
it is $2$-connected if and only if $X_0$ is connected.
Let $$\mathcal{C}_0 (A_n), \widetilde{\mathcal{C}}_0 (A_n), \mathcal{C}_0(\mathcal{T}(A_n)), \mathcal{C}_0(\mathcal{O}(A_n))$$ be the sets of components of $$P_0(A_n), \widetilde{P}_0(A_n), P_0(\mathcal{T}(A_n)), \mathcal{O}_0(A_n)$$ respectively  so that, due to the given notation, we have $$c_0(A_n)=|\mathcal{C}_0 (A_n)|,\
\widetilde{c}_0(A_n)=|\widetilde{\mathcal{C}}_0 (A_n)|,\  c_0(\mathcal{T}(A_n))=|\mathcal{C}_0(\mathcal{T}(A_n))|, \ c_0(\mathcal{O}(A_n))=|\mathcal{C}_0(\mathcal{O}(A_n))|.$$

The following manageable inequalities link the number of components of the graphs in $\mathcal {G}_0.$ They follow immediately from \cite[Propositions 5.4, 5.6]{bis} and \cite[Proposition 3.2]{bub}.
\begin{equation}\label{connected}
\forall n\geq 3,\quad  c_0(\mathcal{O}(A_n))\leq c_0(\mathcal{T}(A_n))\leq \widetilde{c}_0(A_n)=c_0(A_n).
\end{equation}

In this paper, we find in one go the values of $c_0(A_n)= \widetilde {c}_0(A_n), c_0(\mathcal{T}(A_n))$ and $c_0(\mathcal{O}(A_n))$  for all $n\geq 3$, and give information on the components of the graphs in $\mathcal{G}_0$.

The idea of studying $P_0(G)$ through the proper quotient power graph  or through the proper order graph is not a complete novelty.  Even though in the literature it is not made explicit, it is tacitly implicated all the times in which the cyclic subgroups of $G$ or its element orders are used as demonstration tools. For instance, in our language, the formula for computing the number of edges of $P_0(G)$ in
 \cite[Corollary 5]{shi}, which is an application of  \cite[Theorem 4.2]{chagho}, sounds as $|E^*_{P_0(G)}|=\frac{1}{2}\sum_{m\in O_0(G)} s_m(2m-\phi(m)-3)$,
 where $s_m$ denotes the number of elements of order $m$ in $G$ and $\phi$ denotes the Euler totient function.
 It is instead completely original the idea of studying the proper power graph through the proper power type graph.

Let $P$ denote the set of prime numbers. For $b,c\in \mathbb{N}$, put
 \[bP+c=\{x\in \mathbb{N}:  x=bp+c, \ \hbox{for some}\   p\in P\}\]
and define
\begin{equation}\label{A} A=P\cup (P+1)\cup (P+2)\cup(2P)\cup (2P+1).\end{equation}
The set of integers $A$ plays the main role in our connectivity problem for the graphs in  $\mathcal{G}_0$.

\begin{thmA} Let  $P$ be the set of prime numbers and $A$ as in \eqref{A}.

The values of $c_0(A_n)= \widetilde {c}_0(A_n), c_0(\mathcal{T}(A_n))$ and $c_0(\mathcal{O}(A_n))$  are as follows.
\vspace{2mm}

\begin{itemize}

 \item[(i)] For $3\leq n\leq 10,$ they are given in Table 1 below.

 \begin{table}[ht]
\caption{$c_0(A_n),\  \widetilde {c}_0(A_n)$,  $c_0(\mathcal{T}(A_n))$ and $c_0(\mathcal{O}(A_n))$, for $3\leq n \leq 10$.}\label{eqtable1}
\begin{center}
\begin{tabular}{|c|c|c|c|c|c|c|c|c|}
\hline
$n$  & $3$ & $4$ & $5$ & $6$ & $7$ & $8$ & $9$ & $10$\\
\hline
$c_0(A_n)=\widetilde {c}_0(A_n)$ & $1$ & $7$ & $31$ & $121$ & $421$ & $962$ & $5442$ & $29345$\\
\hline
$c_0(\mathcal{T}(A_n))$ & $1$ & $2$ & $3$ & $4$& $4$ & $3$ & $4$ & $3$ \\
\hline
$c_0(\mathcal{O}(A_n))$ & $1$ &$ 2$ &$ 3$ & $3 $&$ 3$ & $2$ & $2$ & $1$\\

\hline
\end{tabular}
\end{center}
\end{table}
\item[(ii)]  For $n\geq 11$, they are given by Table $2$ below.

\begin{table}[ht]
\caption{$c_0(A_n)$,\  $\widetilde {c}_0(A_n),\  c_0(\mathcal{T}(A_n))$ and $c_0(\mathcal{O}(A_{n}))$ for $n\geq 11$.}\label{eqtable2}
\begin{center}
\begin{tabular}{|c|c|c|c|}
\hline
\vspace{-3mm}
& & &\\

{$n\geq 11$} & {$c_0(A_n)=\widetilde {c}_0(A_n)$}&
{$c_0(\mathcal{T}(A_n))$}&
{$c_0(\mathcal{O}(A_{n}))$}
\\
\vspace{-3mm}
& & &\\

\hline
\vspace{-3mm}
& & &\\
$n-2,\, {\frac{n-1}{2}}\in P$ & $\displaystyle{\frac{n(n-1)(n-4)!}{2}+\frac{4n(n-2)(n-4)!}{n-1}+1}$ & $3$ & $2$ \\
\vspace{-3mm}
& & &\\
\hline
\vspace{-3mm}
& & &\\
$n,\, {\frac{n-1}{2}}\in P$ & $\displaystyle{(n-2)!+\frac{4n(n-2)(n-4)!}{n-1}+1}$ & $3$ &$2$ \\
\vspace{-3mm}
& & &\\
\hline
\vspace{-3mm}
& & &\\
$n,\, n-2\in P$ & $\displaystyle{(n-2)!+{\frac{n(n-1)(n-4)!}{2}}+1}$ & $3$ & $3$\\
\vspace{-3mm}
& & &\\

\hline
\vspace{-3mm}
& & &\\
$n-2\in P,\, n\notin P,\,{\frac{n-1}{2}}\notin P$ & $\displaystyle{{\frac{n(n-1)(n-4)!}{2}}+1}$ & $2$ & $2$\\
\vspace{-3mm}
& & &\\

\hline
\vspace{-3mm}
& & &\\

${\frac{n-1}{2}}\in P,\, n\notin P,\, n-2\notin P$ & $\displaystyle{\frac{4n(n-2)(n-4)!}{n-1}+1}$ & $2$ &$1$ \\
\vspace{-3mm}
& & &\\

\hline
\vspace{-3mm}
& & &\\

$n\in P,\,n-2\notin P,\, {\frac{n-1}{2}}\notin P$ & $(n-2)!+1$ & $2$ & $2$\\

\vspace{-3mm}
& & &\\

\hline
\vspace{-3mm}
& & &\\

$n-1\in P,\,{\frac{n}{2}}\notin P$ & $n(n-3)!+1$ & $2$ & $2$ \\
\vspace{-3mm}
& & &\\

\hline
\vspace{-3mm}
& & &\\

$n-1,\,{\frac{n}{2}}\in P$ & $\displaystyle{\frac{4(n-1)(n-3)!}{n}+n(n-3)!+1}$ & $3$ &$2 $\\
\vspace{-3mm}
& & &\\

\hline
\vspace{-3mm}
& & &\\

${\frac{n}{2}}\in P,\, n-1\notin P$ & $\displaystyle{\frac{4(n-1)(n-3)!}{n}+1}$ & $2$ & $1$\\
\vspace{-3mm}
& & &\\

\hline
\vspace{-3mm}
& & &\\

$n\notin A$ &$1$ & $1$ & $1$\\
\hline
\end{tabular}
\end{center}
\end{table}
\end{itemize}
\end{thmA}

We emphasise that the proof of Theorem A is  conduced similarly to the symmetric case in \cite{bis} but it is more tricky and requires some more technicalities and attention.
 The $2$-connectivity of the proper power graph of a finite group $G$ was studied by Doostabadi and Farrokhi in \cite{df} for nilpotent groups, groups admitting a partition, symmetric and alternating groups.
They used ingenious ad hoc arguments for the various classes of finite groups considered, but did not create a general procedure to determine $c_0(G)$. Unfortunately,  the computation of $c_0(A_n)$ in \cite[Theorem 4.7]{df}  is not generally correct, with mistakes related to infinitely many $n$. For instance, they claim that $c_0(A_{21})=19!+1$, while we found that $c_0(A_{21})=21\cdot 10\cdot 17!+1$. The more conspicuous error in \cite{df}  seems to be considering the case $n\in 2P+2$ as significant, while Table $2$ shows it is not.

Going beyond the mere counting, we describe the components of the proper power graph of $A_n$ and their quotients. Note that for $n\notin A$, with $n\neq 3$, no $X_0\in\mathcal{G}_0$ is a complete graph because $X_0$ admits as quotient $\mathcal{O}_0(A_{n})$ which is surely incomplete having as vertices at least two primes.

\begin{thmB} Let  $n\in A$, with $n\geq 11$, and $\widetilde{\Omega}_n$ be the component of $\widetilde{P}_0(A_n)$ containing $[(12)(34)]$.
 \begin{itemize}
  \item[(i)] $\widetilde{P}_0(A_n)$ consists of the main component $\widetilde{\Omega}_n$ and of some isolated vertices of prime order, with at most two primes  involved.
 \item[(ii)] $P_0(A_n)$ consists of the main component induced by  $\{\psi\in A_n\setminus\{id\}: [\psi]\in V_{\widetilde{\Omega}_n}\}$ and of some complete graphs on $p-1$ vertices, with at most two primes $p$ involved.
  \item[(iii)] $P_0(\mathcal{T}(A_n))$ consists of the main component  $\widetilde{t}(\widetilde{\Omega}_n)$ and of some isolated vertices of prime order, with at most two primes  involved.
 \item[(iv)] $\mathcal{O}_0(A_n)$  consists of the main component $\widetilde{o}(\widetilde{\Omega}_n)$ and of some isolated vertices which are primes,  with at most two primes  involved.
\end{itemize}
 In all the above cases, the main component is never complete.
\end{thmB}

Complete  information about the components of the graphs in $\mathcal{G}_0$, for $3\leq n\leq 10$ can be found within the proofs of Theorems \ref{thm:11}, \ref{thm:15}  and \ref{thm:17}, taking into account Lemma \ref{component-quotient} for $P_0(A_n)$.
In particular, looking at the details, one can easily checks that  all the components of $X_0 \in\mathcal{G}_0$ apart from one are isolated vertices (complete graphs when $X_0=P_0(A_n)$)  if and only if $n\geq 11$ or $n=3.$

 \begin{corC} Let  $n\geq 3$  and $A$ as in \eqref{A}. Then the following facts are equivalent:
 \begin{itemize}
 \item[(i)] $P(A_n)$ is $2$-connected;
  \item[(ii)] $P_0(A_n)$ is connected;
 \item[(iii)] $\widetilde{P}_0(A_n)$ is connected;
 \item[(iv)] $P_0(\mathcal{T}(A_n))$ is connected;
 \item[(v)] $n=3$ or $ n\notin A$.
\end{itemize}
 In particular, there exist infinitely many $n$ such that $P(A_n)$ is $2$-connected.
 \end{corC}

\begin{rem}{\rm \begin{itemize}
\item[(a)]$n=16$ is the minimum $n$ for which $P(A_n)$ is $2$-connected with $A_n$ non-abelian simple.
\item[(b)] Exhibiting an infinite family of simple groups with $2$-connected power graph, we give a positive
answer to the question posed by Pourgholi, Yousefi-Azari and Ashrafi \cite[Question 13]{pya}.
\item[(c)] Since Corollary C confirms the role of the proper power type graph in the comprehension of the $2$-connectivity of the power graph for the alternating group,
we wonder if an analogous construction is possible for any simple group.

 \end{itemize}}
\end{rem}
 As recalled before, for the symmetric group the $2$-connectivity
of the order graph is equivalent to the $2$-connectivity of the power graph. This does not happen for
the alternating group.
\begin{corD} Let $n\geq 3.$ Then $\mathcal{O}(A_n)$ is $2$-connected if and only if either $n=3$ or none of $n, n-1$ and $ n-2$ is a prime. The maximum number of components of $\mathcal{O}_0(A_n)$ is $3$ and it is realised if and only if $n=6$ or both $n$ and $n-2$ are prime.
\end{corD}

\vskip 0.4 true cm
\section{ The procedure for computing $c_0(A_n)$ }\label{Preliminaries}
\vskip 0.4 true cm
In this section, we explain the terminology involved in  the procedure which we are going to use for the computation of $c_0(A_n)$. We also present some results about the components of the graphs in $\mathcal{G}_0.$


\subsection{Permutations of a fixed type}

Let  $T=[m_1^{t_1},..., m_k^{t_k}]\in\mathcal{T}(A_n)$.
The number of permutations
 of type $T=[m_1^{t_1},..., m_k^{t_k}]$ in $A_n$ is the same
as the number of this type in $S_n$ and is given by

\begin{equation}\label{count}\mu_T(A_n)=\dfrac{n!}{m_1^{t_1}\cdots m_k^{t_k}t_1!\cdots t_k!}
\end{equation}

while the number $\mu_T[A_n]$ of vertices  of type $T$  in $[A_n]$ coincides with $k_{\widetilde{P}_0(A_n)(T)}$ and thus,
 by \cite[Lemma 6.5 ]{bis}),  is  given by

\begin{equation}\label{countilde}\mu_T[A_n]=\dfrac{\mu_T(A_n)}{\phi(o(T))}.
\end{equation}

\subsection{Components of graphs in $\mathcal{G}_0$}  Let $\mathcal{G}_0$ be the set of graphs defined in \eqref{G0}.
Since we do not want to reduce our research to a mere counting of components, we need to collect some information about the components of the graphs in $\mathcal {G}_0$.

\begin{prop}\label{hom-con}  Let $\varphi:X \rightarrow Y$ be a homomorphism between the graphs $X$ and $Y$.
 \begin{itemize}
\item[(i)] If $\hat{X}$ is a complete subgraph of $X$, then $\varphi(\hat{X})$ is a complete subgraph of $Y$.
\item[(ii)] Let $\varphi$ be pseudo-covering.
\begin{itemize}
\item[(a)] If $C$ is a component of $X$, then $\varphi(C)$ is a component of $Y$.
\item[(b)] For every vertex $x\in V_X$, $\varphi(C_{X}(x))=C_{Y}(\varphi(x))\cong C_{X}(x)/\hspace{-1mm}\sim_{\varphi}$.
\end{itemize}
\end{itemize}
 \end{prop}
 \begin{proof}  (i) is straightforward; (ii)  comes from \cite[Proposition 5.11]{bub}.
 \end{proof}

 \begin{lem}\label{component-quotient}  Let $\pi$ be the projection of $ P_0(A_n)$ on its quotient $\widetilde{P}_0(A_n) $.  Then,  the map from $\mathcal{C}_0 (A_n)$ to $\widetilde{\mathcal{C}}_0 (A_n) $ which associates to every $C\in \mathcal{C}_0 (A_n),$ the component $\pi(C) $ is a bijection. Given $\widetilde{C}\in \widetilde{\mathcal{C}}_0 (A_n) $, the set  of vertices  of the unique $C\in \mathcal{C}_0 (A_n)$ such that $\pi(C)=\widetilde{C}$ is given by $\pi^{-1}(V_{\widetilde{C}})$.
\end{lem}

\begin{proof} Apply \cite[Lemma 3.7]{bis} to the group $A_n.$
\end{proof}

The above lemma says that, once the components of $\widetilde{P}_0(A_n)$ are known, it is immediate to recover those of $P_0(A_n).$ Moreover it leads to a useful result about isolated vertices in the graphs belonging to $\mathcal{G}_0$.

\begin{lem} \label{isolated}\noindent \begin{itemize}\item[(i)] The type
$T\in \mathcal{T}_0(A_n)$ is isolated in $P_0(\mathcal{T}(A_n))$ if and only
if each $[\psi]\in [A_n]_0$ of type $T$ is isolated in $\widetilde{P}_0(A_n)$.
\item[(ii)] If $m\in O_0(A_n)$ is isolated in $\mathcal{O}_0(A_n),$ then each vertex of
order $m$ is isolated in $\widetilde{P}_0(A_n)$.
\item[(iii)] If, for some $\psi\in A_n$, $[\psi]$ is
isolated in $\widetilde{P}_0(A_n)$, then $o(\psi)$ is a prime $p$ and the component of $P_0(A_n)$ containing $\psi$ is a complete graph on $p-1$ vertices.
\end{itemize}
\end{lem}
\begin{proof}(i) follows from \cite [Proposition 6.3]{bis};  (ii) and  (iii) follow from \cite [Lemma  4.5]{bis}.
\end{proof}
Let $C\in \widetilde{\mathcal{C}}_0 (A_n)$ and $T\in \mathcal{T}_0(A_n)$.  Applying the definitions given in Section \ref{components-intro} to the graphs $X=\widetilde{P}_0(A_n)$ and $Y=P_0(\mathcal{T}(A_n))$ and to the homomorphism $\widetilde{t}$, we have that $k_{C}(T)$ is the multiplicity of $T$ in $C$, that is, the number of vertices $[\psi]\in V_C$ such that $T_{\psi}=T$; $\widetilde{\mathcal{C}}_0(A_n)_{T}$ is the set of components in $\widetilde{\mathcal{C}}_0 (A_n)$ which are admissible for $T$.
Put $\widetilde{c}_0(A_n)_{T}=|\widetilde{\mathcal{C}}_0(A_n)_{T}|.$
By \cite[Lemma 6.4]{bis} we deduce immediately the following result.

\begin{lem}\label{C(T)} Let $T\in \mathcal{T}_0(A_n)$
and $C\in \widetilde{\mathcal{C}}_0(A_n)_{T}$. Then the following facts hold:
\begin{itemize}\item[(i)] $\mathcal{T}(C)=V_{C(T)}$;
\item[(ii)] $\widetilde{o}(C)$ is a connected subgraph of the component of $\mathcal{O}_0(A_n)$ containing $o(T)$.
\end{itemize}
\end{lem}

 From \cite[Lemma 6.6]{bis}, we know that
\begin{equation}\label{term} \widetilde{c}_0(A_n)_{T}={\frac{\mu_{T} [A_n]}{k_{C}(T)}}.\end{equation}
By~\cite[Theorem A ]{bis} applied to the alternating group, we get the
following reediting of the Procedure ~\cite[6.10]{bub}  for counting
the number of components of $\widetilde{P}_0(A_n)$.
\vskip 0.4 true cm

 \begin{proc} {\bf Procedure to compute $\widetilde {c}_0(A_n)$ } \label{procedureA_n}
{\rm

\begin{itemize}\item[ I)] {\it  Selection of types $T_i$ and components $C_i$}
\end{itemize}
\begin{itemize}
\item[ {\it Start}] {\rm :  Pick  $T_1\in \mathcal{T}_0(A_n)$ and choose any
$C_1\in \widetilde{\mathcal{C}}_0(A_n)_{T_1}$.}

\item[ {\it Basic step}]{\rm :  Given $T_1,\dots, T_i\in \mathcal{T}_0(A_n) $
and $C_1,\dots,C_i\in\widetilde{\mathcal{C}}_0(A_n)$ such
that $C_j\in \widetilde{\mathcal{C}}_0(A_n)_{T_j}\ (1\leq j\leq i)$,
choose any $T_{i+1}\in \mathcal{T}_0(A_n) \setminus \bigcup_{j=1}^i \mathcal{T}(C_{j})$
and any $C_{i+1}\in \widetilde{\mathcal{C}}_0(A_n)_{T_{i+1}}.$}

\item[ {\it Stop}]{\rm: The procedure stops in $c_0(\mathcal{T}(A_n))$ steps.}
\end{itemize}
\medskip

\begin{itemize}\item[ II)] {\it  The value of  $\widetilde {c}_0(A_n)$ }
\end{itemize}
{\rm Compute the integers
\begin{equation}\label{term2} \widetilde{c}_0(A_n)_{T_j}={\frac{\mu_{T_j} [A_n]}{k_{C_j}(T_j)}} \quad (1\leq j\leq c_0(\mathcal{T}(A_n))
\end{equation}
 and sum them up to get  $\widetilde {c}_0(A_n)$. }}
\end{proc}

The complete freedom in the choice of the
$C_j\in \widetilde{\mathcal{C}}_0(A_n)_{T_j}$  allows us
to compute  each $ \widetilde{c}_0(A_n)_{T_j}={\frac{\mu_{T_j} [A_n]}{k_{C_j}(T_j)}}$,
selecting $C_j$ as the component containing $[\psi]$, for $[\psi]$ chosen
as preferred among the $\psi\in A_n$ with $T_{\psi}=T_j.$ We will
apply  this fact with no further mention.  We emphasize also that the computing of $\mu_{T_j} [A_n]$ is made easy by \eqref{countilde} and \eqref{count}.
Remarkably, the number
$c_0(\mathcal{T}(A_n))$ counts the steps of the procedure.

\section{ Some small degrees}

\vskip 0.4 true cm
In this section, we point out some general properties  of the graphs in $\mathcal {G}_0$ and use them to determine $c_0(A_n)=\widetilde {c}_0(A_n)$, $c_0(\mathcal{T}(A_n))$ and $c_0(\mathcal{ O}(A_n)),$ for $3\leq n\leq 7$. Since
$[A_3]_0=\{[(1~2~3)]\}$, we trivially  have $\widetilde{c}_0(A_3)=c_0(\mathcal{T}(A_3))=c_0(\mathcal{O}(A_3))=1$.
\begin{lem}\label{p-isolated} Let $p$ be a prime number and $n\in\{p, p+1,p+2\}$, with $n\geq 4$. Then the following facts hold:
\begin{itemize}
\item[(i)] $p\in O_0(A_n)$ and is isolated in $\mathcal{O}_0(A_n).$ If $[\psi]\in[A_n]_0$ has order $p$, then $[\psi]$
is isolated in $\widetilde{P}_0(A_n)$ and $T_{\psi}$ is isolated in $P_0(\mathcal{T}(A_n))$;
\item[(ii)] the number $c_p$ of components of $\widetilde{P}_0(A_n)$ containing elements of order $p$ is given by the following table.
 \begin{table}[ht]
\caption{}
\label{eqtable3}
\begin{center}
\begin{tabular}{|c|c|}
\hline
$n$  & $c_p$ \\
\hline
\vspace{-3mm}
&\\
$p$ & $(p-2)!$ \\
\vspace{-3mm}
&\\
\hline
\vspace{-3mm}
&\\
$p+1$ & $(p+1)(p-2)!$  \\
\vspace{-3mm}
&\\
\hline
\vspace{-3mm}
&\\
$p+2$, $p$ \hbox {odd} & $\frac{(p+2)(p+1)(p-2)!}{2}$ \\
\vspace{-3mm}
&\\
 \hline
$4$ & $3$\\
\hline
\end{tabular}
\end{center}
\end{table}

\end{itemize}
\end{lem}
\begin{proof} (i)  By the assumptions on $n$, we get $p\mid \frac{n!}{2}$ so that $A_n$ admits elements of order $p$, but there exists no element  in $A_n$ with order $kp$ for $k\geq 2$. This says that $p$ is isolated in $\mathcal{O}_0(A_n).$ Then, by Lemma \ref{isolated},  $[\psi]$
is isolated in $\widetilde{P}_0(A_n)$ and $T_{\psi}$ is isolated in $P_0(\mathcal{T}(A_n)),$ for all $[\psi]\in[A_n]$ with $o(\psi)=p.$

(ii)-(iv) We consider the elements of order $p$ in $A_n$.
If $n=p$, they are those of type $[p];$ if $n=p+1$,  those of type $[1, p];$  if $n=p+2$,  those of type $[1^2 ,p]$ if $p$ is odd and those of type $[2^2]$ if $p=2$. So the counting follows from \eqref{term} through \eqref{count} and \eqref{countilde}.

\end{proof}
\begin{thm}\label{thm:11} For $3\leq n \leq 7,$ the values of $c_0(A_n)=\widetilde {c}_0(A_n)$, $c_0(\mathcal{T}(A_n))$ and $c_0(\mathcal{O}(A_n))$ are given by the following table.
\begin{table}[ht]
\caption{$c_0(A_n)=\widetilde {c}_0(A_n)$ and $c_0(\mathcal{T}(A_n))$, for $3\leq n \leq 7$.}\label{eqtable4}
\begin{center}
\begin{tabular}{|c|c|c|c|c|c|}
\hline
$n$  & $3$ & $4$ & $5$ & $6$ & $7$\\
\hline
$c_0(A_n)=\widetilde {c}_0(A_n)$ & $1$ & $7$ & $31$ & $121$ & $421$\\
\hline
$c_0(\mathcal{T}(A_n))$ & $1$ & $2$ & $3$ & $4$& $4$\\
\hline
$c_0(\mathcal{O}(A_n))$ & $1$ &$ 2$ &$ 3$ & $3 $&$ 3$\\

\hline
\end{tabular}
\end{center}
\end{table}
\end{thm}
\begin{proof} Let $G=A_n$, for $3\leq n \leq 7$, acting on $N=\{1,\dots,n\}$. We compute $\widetilde{c}_0(A_n)$ separately for each degree. Since $\widetilde{P}_0(A_4)$ has only three vertices of type $[2^2]$ and four of type $[1, 3]$ and all of them are isolated, we have $\widetilde{c}_0(A_4)=7$, $c_0(\mathcal{T}(A_4))=c_0(\mathcal{O}(A_4))=2$.

In $\widetilde{P}_0(A_5)$, by Lemma \ref{p-isolated}, all the vertices of type $[1^2, 3]$ and $[5]$ are isolated.  Moreover it is immediately  checked  that also the vertices of type $[1, 2^2]$ are isolated. In other words the graph $\widetilde{P}_0(A_5)$ is totally disconnected.
By \eqref{term} and \eqref{count}, applying the procedure \ref{procedureA_n},
 we then  get $$\widetilde{c}_0(A_5)=\mu_{[1, 2^2]}[A_5]+\mu_{[1^2, 3]}[A_5]+\mu_{[5]}[A_5]=31,$$
 while $c_0(\mathcal{T}(A_5))=3$. Since $O_0(A_5)=\{2,3,5\}$ we also have $c_0(\mathcal{O}(A_5))=3.$

In $\widetilde{P}_0(A_6)$, by Lemma~\ref{p-isolated}, all the vertices of type $T_1=[1, 5]$ are contained in $\widetilde{c}_0(A_6)_{T_1}=36$ components which are isolated vertices. Also, having in mind~\cite[Corollary 5.5]{bis}, it is clear that all the vertices of types $T_2=[1^3, 3]$ and $T_3=[3^2]$ are isolated. Thus  by \eqref{term} and \eqref{count}, we get $\widetilde{c}_0(A_6)_{T_2}=\widetilde{c}_0(A_6)_{T_3}=20. $ Let $T_4=[1^2,2^2]$ and consider the component $C_4$ containing $\psi=(1~2)(3~4)$. Since $T_4$ admits no proper power, the edges incident to $[\psi]$ are given only by $\{[\psi],[\varphi]\}$ for those $\varphi\in A_6$ such that $\varphi^2=\psi.$  In particular $\varphi=(a~b)(c~d~e~f)$, where $\{a, b, c, d, e, f\}=N$  and  $(c~e)(d~f)=(1~2)(3~4)$. It follows that $(a~b)=(5~6)$, while either $(c~d~e~f)=(1~3~2~4)$ or $(c~d~e~f)=(3~1~4~2)$. On the other hand, by~\cite[Lemma 5.2]{bis}, the type   $T_{\varphi}=[2,4]$ is not a proper power, because its terms are distinct. Thus $C_4$ is the path:
$$[(1~3~2~4)(5~6)], [(1~2)(3~4)], [(3~1~4~2)(5~6)]$$
and $k_{C_4}(T_4)=1.$ So, by \eqref{term},  $\widetilde{c}_0(A_6)_{T_4}=\mu_{[1^2, 2^2]}[A_6]=45. $
Since all the types admissible for $A_6$ have appeared, the application of the procedure \ref{procedureA_n} gives $\widetilde{c}_0(A_6)=121$, while $c_0(\mathcal{T}(A_6))=4$. Moreover $c_0(\mathcal{O}(A_6))=3$ and the components of $\mathcal{O}_0(A_6)$ have vertex sets $\{3\}$, $\{5\}$ and $\{2,4\}.$

In $\widetilde{P}_0(A_7)$, again by Lemma~\ref{p-isolated}, all the vertices of type $T_1=[1^2, 5]$ and $T_2=[7]$ are isolated  and $\widetilde{c}_0(A_7)_{T_1}=126, \widetilde{c}_0(A_7)_{T_2}=120.$ Moreover, by ~\cite[Corollary 5.5]{bis}, it is clear that all the vertices of types $T_3=[1, 3^2]$  are isolated and, by \eqref{term} and \eqref{count},
we get $\widetilde{c}_0(A_7)_{T_3}=140.$ Let $C_i$ for $i\in \{1,2,3\}$ be a component admissible for $T_i$. Consider now the type $T_4=[1^4, 3]$
and the component $C_4$ containing $\psi=(1~2~3)$. Since $T_4$ admits no proper power, the edges incident to $[\psi]$ are given only
by $\{[\psi],[\varphi]\}$ for those $\varphi\in A_7$ such that $T_{\varphi}=[2^2,3]$ and $\varphi^2=\psi.$ Then  $\varphi=(a~b)(c~d)(e~f~g)$,
where $\{a, b, c, d, e, f,g\}=N$  and  $(e~g~f)=(1~2~3)$, $\{a, b, c, d\}=\{4, 5, 6, 7\} .$ It follows that there are three choices for $[\varphi]$.
Moreover each of those $[\varphi]$ is adjacent to exactly one vertex of type $[1^3, 2^2]$, which in turn  is adjacent to three vertices of
type $[1, 2, 4]$ and these elements are not adjacent to any other vertex.
In particular we have $k_{C_4}(T_4)=1$ and thus, by \eqref{term}, $\widetilde{c}_0(A_7)_{T_4}=\mu_{[1^4, 3]}[A_7]=35$. Since $\bigcup_{i=1}^4\mathcal{T}(C_i)=\mathcal{T}(A_7)$, the procedure  \ref{procedureA_n}  closes giving $\widetilde{c}_0(A_7)=421$ and $c_0(\mathcal{T}(A_7))=4$. Moreover $c_0(\mathcal{O}(A_7))=3$ and the three components of $\mathcal{O}_0(A_7)$ have vertex sets $\{2, 3, 4, 6
\}$, $\{5\}$ and $\{7\}.$
\end{proof}
 \vskip 0.4 true cm

\section{The degrees $8$ and $9$}

\vskip 0.4 true cm
In this section, we determine $c_0(A_n)=\widetilde {c}_0(A_n)$, $c_0(\mathcal{T}(A_n))$ and $c_0(\mathcal{ O}(A_n))$ for $n\in\{8,9\}.$ We also obtain some results to treat all the cases with $n\geq8.$ For every $\psi\in A_n$, we denote by $M_{\psi}=\{i\in N :\psi(i)\neq i\}$ the support of $\psi.$
\begin{defn}\label{def1} For $n\geq 8$, denote by $\widetilde{\Omega}_n$ the component of $\widetilde{P}_0(A_n)$ containing the vertex $[(12)(34)]$.
\end{defn}

\begin{lem}\label{lem:12}
Let $n\geq 8$.
\begin{itemize} \item[(i)] All the vertices of $\widetilde{P}_0(A_n)$ of  type $[1^{n-4}, 2^2]$ belong to $\widetilde{\Omega}_n$.
\item[(ii)] $\mathcal{T}(\widetilde{\Omega}_n) \supseteq \{[1^{n-4}, 2^2],[1^{n-3},3], [1^{n-7},2^2,3] \}$.
\item[(iii)] For every $T\in \mathcal{T}(\widetilde{\Omega}_n)$,
$\widetilde{\Omega}_n$ contains all the vertices of $\widetilde{P}_0(A_n)$ of type $T$ and $\mathcal{T}(\widetilde{\Omega}_n)=V_{C(T)}$.
\end{itemize}

\end{lem}

\begin{proof}
(i) Let  $[\pi_1]$ and $[\pi_2]$ be distinct vertices in $[A_n]_0$ of type $[1^{n-4}, 2^2]$. Then $\pi_1, \pi_2\in A_n$ are distinct and share at most one transposition. Let $\pi_1=(a~b)(c~d)$ for suitable distinct $a, b, c, d\in N.$  We  analyze the two possibilities:
\begin{itemize}
\item[(1)] $\pi_1$ and $\pi_2$ share a transposition;
\item[(2)]$\pi_1$ and $\pi_2$ share no transposition.
\end{itemize}
(1) In this case we have $\pi_2=(a~b)(e~f)$, for suitable  distinct $ e, f\in N$ and $\{c,d\}\neq \{e,f\}$. If $|\{c,d\}\cap \{e,f\}|=1,$ we may assume that $e=c$, that is, $\pi_2=(a~b)(c~f)$. Since $n\geq 8$, there exist $x,y,z \in N\setminus \{a,b,c,e,f\}$
and we have  the following path  between $[\pi_1]$ and $[\pi_2]$:\\
$[\pi_1], [(a~b)(c~d)(x~y~z)], [(x~y~z)], [(a~b)(c~f)(x~y~z)], [\pi_2].$

If $|\{c,d\}\cap \{e,f\}|=0,$ let $\pi=(a~b)(c~e)$. By the previous case, there is a path between $[\pi_1]$ and $[\pi]$ as well as a path between $[\pi_2]$ and $[\pi]$ and so also a path
 between $[\pi_1]$ and $[\pi_2]$.

(2) We distinguish the three possible subcases:
\begin{itemize}
\item[(a)] $|M_{\pi_1}\cap M_{\pi_2}|=2$. Then $\pi_2=(a~e)(c~f)$, for suitable $e,f\in N,$ such that $\{a,e,c,f\}\cap \{a,b,c,d\} =\{a,c\}$.
Let $\pi=(a~b)(c~f)$. Then, by  1), there exists a path between $[\pi_1]$ and $[\pi]$ and a path between $[\pi_2]$ and $[\pi]$. So there exists a path also between $[\pi_1]$ and $[\pi_2]$.
\item[(b)] $|M_{\pi_1}\cap M_{\pi_2}|=1$. Then $\pi_2=(a~e)(f~g)$, for suitable $e ,f,g\in N$ such that $\{a,e,f,g\}\cap \{a,b,c,d\} =\{a\}$.
Now consider $\pi=(a~b)(f~g)$ and argue as in (a).
\item[(c)] $|M_{\pi_1}\cap M_{\pi_2}|=0$. Then $\pi_2=(e~f)(g~h)$ for suitable $e ,f,g,h\in N$ such that $\{e,f,g,h\}\cap \{a,b,c,d\} =\varnothing$.
Now consider $\pi=(a~b)(e~f)$ and argue as in (a).
\end{itemize}
This shows that all the vertices of $\widetilde{P}_0(A_n)$  of type $[1^{n-4}, 2^2]$ lie  in the same component $\widetilde{\Omega}_n.$

(ii)  Collecting the types met in the paths used for i), we immediately get  $$\mathcal{T}(\widetilde{\Omega}_n)\supseteq \{[1^{n-4},2^2], [1^{n-3},3], [1^{n-7},2^2,3]\}.$$

(iii)  Let  $T\in \mathcal{T}(\widetilde{\Omega}_n)$. The fact that $\mathcal{T}(\widetilde{\Omega}_n)=V_{C(T)}$ is an application of Lemma \ref{C(T)}.  The fact that $\widetilde{\Omega}_n$ contains all the vertices in $[A_n]_0$ of type $T$  is instead a consequence of (i) and of \cite[Corollary 6.7(iii)]{bis}.
\end{proof}

\begin{cor} \label{complete} Let $n\geq 8$ and $\Omega_n$ be the unique component of $P_0(A_n)$ such that $\pi (\Omega_n)=\widetilde{\Omega}_n.$ Then neither one of the components $\Omega_n,\  \widetilde{\Omega}_n,\  \widetilde{t}(\widetilde{\Omega}_n)$ of the graphs $P_0(A_n)$, $\widetilde{P}_0(A_n)$, $P_0(\mathcal{T}(A_n))$
respectively, nor  the connected subgraph $\widetilde{o}(\widetilde{\Omega}_n)$ of $\mathcal{O}_0(A_n)$ is a complete graph. 
\end{cor}
\begin{proof}  First of all note that  the existence of a unique component  $\Omega_n$ of  $P_0(A_n)$ such that $\pi(\Omega_n)=\widetilde{\Omega}_n$ is guaranteed by Lemma \ref{component-quotient}. Moreover, by Proposition \ref{hom-con}, we have that  $\widetilde{t}(\widetilde{\Omega}_n)$ is a component of $\widetilde{P}_0(A_n)$ with  $V_{\widetilde{t}(\widetilde{\Omega}_n)}=\mathcal{T}(\widetilde{\Omega}_n)$. On the other hand, by \cite[Proposition 5.6 ]{bis}, the  map
$o_{\mathcal{T}}:\mathcal{T}_0(A_n)\rightarrow O_0(A_n)$ defined by $o_{\mathcal{T}}(T)=o(T)$ for all $T\in \mathcal{T}_0(A_n)$, defines a complete $2$-homomorphism $o_{\mathcal{T}}:P_0(\mathcal{T}(A_n))\rightarrow \mathcal{O}_0(A_n)$ such that $o_{\mathcal{T}}\circ \widetilde{t}=\widetilde{o}$, which gives
$\widetilde{o}(\widetilde{\Omega}_n)=o_{\mathcal{T}}(\widetilde{t}(\widetilde{\Omega}_n)).$
It follows that we can interpret the sequence  of graphs $$\Omega_n,\  \widetilde{\Omega}_n,\  \widetilde{t}(\widetilde{\Omega}_n),\ \widetilde{o}(\widetilde{\Omega}_n)$$ as \begin{equation}\label{new}\Omega_n,\  \pi(\Omega_n),\  \widetilde{t}( \pi(\Omega_n)),\ o_{\mathcal{T}}(\widetilde{t}( \pi(\Omega_n))).\end{equation}
It is immediate to check that $\widetilde{o}(\widetilde{\Omega}_n)$ is not a complete graph because, by Lemma \ref{lem:12}, it admits as vertices the integer $2$ and $3$ and no edge exists between them in $\mathcal{O}_0(A_n)$. Then to deduce that no graph in the sequence \eqref{new} is complete, we start from the bottom and apply three times Proposition \ref{hom-con}\,(i). 
\end{proof}
Note that, in principle, $\widetilde{o}(\widetilde{\Omega}_n)$ is not a component of $\mathcal{O}_0(A_n)$ because $\widetilde{o}$ is not, in general,  pseudo-covering. For instance $\widetilde{o}(\widetilde{\Omega}_9)$ is not a component of $\mathcal{O}_0(A_9)$ because $3$ is a vertex of $\widetilde{o}(\widetilde{\Omega}_9)$ but $9$, adjacent to $3$ in $\mathcal{O}_0(A_9)$, is not a vertex of  $\widetilde{o}(\widetilde{\Omega}_9)$. That  fact is easily understood looking inside the proof of Theorem \ref{thm:15}.
Anyway we will see, along the proof of Theorem A, that $\widetilde{o}(\widetilde{\Omega}_n)$ is indeed a component at least  for $n\geq 11$.

\begin{lem}\label{lem:13}
Let $n\in\{8,9\}$.
\begin{itemize}
 \item[(i)] The vertices of  $\widetilde{P}_0(A_n)$ of type $[1^{n-8}, 2^4]$ belong to the same component  $\widetilde{\Lambda}_n$ of $\widetilde{P}_0(A_n)$.
\item[(ii)] $\mathcal{T}(\widetilde{\Lambda}_n)= \{[1^{n-8},2^4], [1^{n-8},2,6], [1^{n-6},3^2],[1^{n-8},4^2]\}$.
\item[(iii)] $\widetilde{\Lambda}_n\neq \widetilde{\Omega}_n$.
\end{itemize}
\end{lem}
\begin{proof} Let first $n=8$. Let $[\pi_1]$ and $[\pi_2]$ be distinct elements in $[A_8]_0$ of type $[2^4]$, with $\pi_1=(a~b)(c~d)(e~f)(g~h)$, where $\{a, b, c, d, e, f, g, h\}=N.$
Since $\pi_1, \pi_2$ are distinct, they share at most two transpositions. We  analyze the three possibilities:
\begin{itemize}
\item[(1)] $\pi_1$ and $\pi_2$ share one transposition;
\item[(2)] $\pi_1$ and $\pi_2$ share two transposition;
\item[(3)]$\pi_1$ and $\pi_2$ share no transposition.
\end{itemize}
(1)  In this case, since the $2$-cycles in which $\pi_1$ splits commute and also the entries in each cycle commute,  we can assume that $\pi_2=(a~b)(c~e)(h~f)(g~d)$ and look to the  path
$$[\pi_1], [(a~b)(c~e~g~d~f~h)], [(c~g~f)(e~d~h)], [(c~h~g~e~f~d)(a~b)], [\pi_2].$$

\noindent (2) Here, for the same considerations as in (1), we reduce to\newline $\pi_2=(a~b)(c~d)(e~g)(f~h)$. Consider $\pi=(a~b)(c~e)(d~h)(f~g)$. Then by (1), there exists a path between $[\pi_1]$ and $[\pi]$ and a path between $[\pi_2]$ and $[\pi]$. Thus there exists also a path between $[\pi_1]$ and $[\pi_2]$.\\

\noindent (3)  Here we reduce to the two cases $\pi_2=(a~c)(b~d)(e~g)(f~h)$ or $\pi_2=(a~c)(b~g)(e~d)(f~h)$. In the first case let $\pi=(a~b)(c~d)(e~g)(f~h)$. Then by (2), there exists a path between $[\pi_1]$ and $[\pi]$ and a path between $[\pi_2]$ and $[\pi]$. In the second case let $\pi=(a~b)(c~g)(e~d)(f~h)$. Then by (1), there exists a path between $[\pi_1]$ and $[\pi]$ and by (2), there exists a path between $[\pi_2]$ and $[\pi]$. In both cases we then get a path between $[\pi_1]$ and $[\pi_2]$.

This shows that all vertices of type $[2^4]$ in $[A_8]_0$ are in the same component $\widetilde{\Lambda}_8$ of $\widetilde{P}_0(A_8)$. Collecting the types met along the considered paths we also  get  $\mathcal{T}(\widetilde{\Lambda}_8)\supseteq\{[2^4], [2,6], [1^{2},3^2]\}.$ It is also clear that $[4^2]\in \mathcal{T}(\widetilde{\Lambda}_8)$ because $[(1~2~3~4)(5~6~7~8)]^2=(1~3)(2~4)(5~7)(6~8)$. To show that  $\mathcal{T}(\widetilde{\Lambda}_8)=\{[2^4], [2,6], [1^{2},3^2],[4^2]\}$ it is enough to observe that the set of types $\{[2^4], [2,6], [1^{2},3^2],[4^2]\}$ is closed by proper powers within the types admissible for $A_8$. Finally, since the type $[1^4,2^2]$ is admissible for $\widetilde{\Omega}_8$ but not for $\widetilde{\Lambda}_8$ we deduce that $\widetilde{\Lambda}_8\neq \widetilde{\Omega}_8$.

Next let $n=9$ and let $[\pi_1]$, $[\pi_2]$ be distinct elements in $[A_9]_0$ of type $[1, 2^4]$. If $\pi_1,\pi_2$ have the same fixed point we are inside a copy of $A_8$ and so there is a path between $[\pi_1]$ and $[\pi_2]$ in $\widetilde{P}_0(A_8)$, which is also a path in $\widetilde{P}_0(A_9).$ If $\pi_1,\pi_2$ have a different fixed point, without loss of generality we may assume that $\pi_1$ fixes $9$ and $\pi_2$ fixes $8$.
Now consider $\varphi_1=(1~2)(3~4)(5~6)(7~8)$ which fixes $9$ and $\varphi_2= (1~2)(3~4)(5~6)(7~9)$ which fixes $8$. We have the following path between $[\varphi_1]$ and $[\varphi_2]$:
$$[\varphi_1], [(1~3~5~2~4~6)(7~8)], [(1~5~4)(3~2~6)], [(1~3~5~2~4~6)(7~9)],[\varphi_2].$$

Since $\pi_i$ and $\varphi_i$ have the same fixed point, there is a path between $[\pi_i]$ and $[\varphi_i]$, for all $i=1,2$ and thus also a path between $[\pi_1]$ and $[\pi_2]$. This shows that all the vertices of type $[1,2^4]$ in $[A_9]_0$ are in the same component $\widetilde{\Lambda}_9$ of $\widetilde{P}_0(A_9)$. Collecting the types met along the paths we get $\mathcal{T}(\widetilde{\Lambda}_9)\supseteq\{[1,2^4], [1,2,6], [1^{3},3^2]\}.$ It is also clear that $[1,4^2]\in \mathcal{T}(\widetilde{\Lambda}_9)$ because $[(1~2~3~4)(5~6~7~8)]^2=(1~3)(2~4)(5~7)(6~8)$. Arguing as for $n=8$, we then get  $$\mathcal{T}(\widetilde{\Lambda}_9)=\{[1,2^4], [1,2,6], [1^{3},3^2],[1,4^2]\}.$$ Finally, since the type $[1^5,2^2]$ is admissible for $\widetilde{\Omega}_9$ but not for $\widetilde{\Lambda}_9$ we deduce that $\widetilde{\Lambda}_9\neq \widetilde{\Omega}_9$.

\end{proof}
\begin{thm}\label{thm:15} Let $n\in\{8, 9\}$. Then
the values of $c_0(A_n)=\widetilde {c}_0(A_n)$, $c_0(\mathcal{T}(A_n))$ and $c_0(\mathcal{O}(A_n))$ are given in Table \ref{eqtable5} below.
\begin{table}[ht]
\caption{$c_0(A_n)=\widetilde {c}_0(A_n)$, $c_0(\mathcal{T}(A_n))$ and $c_0(\mathcal{O}(A_n)),$ for $n=8, 9$.}\label{eqtable5}
\begin{center}
\begin{tabular}{|c|c|c|}
\hline
$n$  & $8$ & $9$\\
\hline
$c_0(A_n)=\widetilde {c}_0(A_n)$ & $962$ & $5442$\\
\hline
$c_0(\mathcal{T}(A_n))$ & $3$ & $4$\\

\hline
$c_0(\mathcal{O}(A_n))$ & $2$ & $2$\\

\hline
\end{tabular}
\end{center}
\end{table}
\end{thm}

\begin{proof}
Let $n=8$. By Lemma~\ref{p-isolated}, the prime $7$ is isolated in $\mathcal{O}_0(A_8).$ It follows that also
$T_1=[1, 7]$ is isolated in $\widetilde{P}_0(A_8)$ and thus $\widetilde {c}_0(A_8)_{T_1}=960.$ By Lemma~\ref{lem:12},
all the vertices  of $\widetilde{P}_0(A_8)$ of type $T_2=[1^4, 2^2]$ are in $\widetilde{\Omega}_8$, so that $\widetilde {c}_0(A_8)_{T_2}=1$,
and $\mathcal{T}(\widetilde{\Omega}_8) \supseteq \{[1^{4}, 2^2],[1^{5},3], [1,2^2,3] \}.$
Moreover looking at the types admissible for $A_8$ which can have as proper power one of the types in $\{[1^{4}, 2^2],[1^{5},3], [1,2^2,3] \}$,
we get immediately that $$\mathcal{T}(\widetilde{\Omega}_8) =\{[1^{4}, 2^2],[1^{5},3], [1,2^2,3] , [3, 5], [1^3, 5], [1^2, 2, 4]\}.$$
Thus $$V_{\widetilde{o}(\widetilde{\Omega}_8)} =\{2,3,4,5,6,15\},$$ so that $\widetilde{o}(\widetilde{\Omega}_8)$ is included in the component of $\mathcal{O}_0(A_8)$ containing $2$.

By Lemma ~\ref{lem:13}, all the vertices of type $T_3=[2^4]$ are in $\widetilde{\Lambda}_8$ and
$$\mathcal{T}(\widetilde{\Lambda}_8)=\{[2^4], [2,6], [1^{2},3^2],[4^2]\},\quad \widetilde{o}(\widetilde{\Lambda}_8)=\{2,3,4,6\} $$
so that $\widetilde {c}_0(A_8)_{T_3}=1$ and the procedure \ref{procedureA_n} closes giving $c_0(A_8)=962$. The discussion above shows also that $c_0(\mathcal{T}(A_8))=3$ and $c_0(\mathcal{O}(A_8))=2.$

Let $n=9$. By Lemma~\ref{p-isolated}, the prime $7$ is isolated in $\mathcal{O}_0(A_9).$ Thus, for $T_1=[1^2, 7]$  we have $\widetilde {c}_0(A_9)_{T_1}=4320$.  By Lemma~\ref{lem:12},
all the vertices of type $T_2=[1^5, 2^2]$ are in $\widetilde{\Omega}_9$ and $\mathcal{T}(\widetilde{\Omega}_9) \supseteq \{[1^{5}, 2^2],[1^{6},3], [1^2,2^2,3] \}.$ In particular  $\widetilde {c}_0(A_9)_{T_2}=1$.
Moreover, looking at the types admissible for $A_9$ which can have as proper power one of the types in $\{[1^{5}, 2^2],[1^{6},3], [1^2,2^2,3] \},$
we get  that $$\mathcal{T}(\widetilde{\Omega}_9) =\{[1^{5}, 2^2],[1^{6},3], [1^2,2^2,3] , [1,3, 5], [1^4, 5], [1^3, 2, 4]\}$$
 and so $$V_{\widetilde{o}(\widetilde{\Omega}_9)} =\{2,3,4,5,6,15\}.$$

 By Lemma ~\ref{lem:13}, all the vertices of type $T_3=[1,2^4]$ are in $\widetilde{\Lambda}_9$ and
$$\mathcal{T}(\widetilde{\Lambda}_9)=\{[1,2^4], [1,2,6], [1^{3},3^2],[1,4^2]\},\quad V_{\widetilde{o}(\widetilde{\Lambda}_9)}=\{2,3,4,6\}$$
so that $\widetilde {c}_0(A_9)_{T_3}=1$.

Finally consider the type $T_4=[3^3]$. Let  $\varphi\in A_9$ with $T_{\varphi}=T_4$ and let $C_4$ be the component of $\widetilde{P}_0(A_9)$ containing $[\varphi]$. We show that $k_{C_4}(T_4)=1$.
Note, first of all, that $T_4$ has no proper power and that it is the proper power only of $[9].$ Moreover, by~\cite[Lemma 5.2]{bis}, the type $[9]$
is not a proper power and its only proper power has exponent $a\in\mathbb{N}$, with $1<a<9$ and $a$ not coprime to $9$, so that $a=3$.
Thus, by Lemma \ref{C(T)}, we have $V_{C(T_4)}=\{T_4,[9]\}=\mathcal{T}(C_4)$ and so also $V_{ \widetilde{o}(C_4)}=\{3,9\}.$

 We show that $[\varphi]$ is the only vertex of type $T_4$ in $C_4$, showing that there exists no path between $[\varphi]$ and some $[\sigma]\in [A_9]_0$ such that $[\sigma]\neq [\varphi]$ and $T_{\sigma}=T_4.$ Let $\gamma$ be a path of length at least $1$ with end vertex $[\varphi]$. Since the only types admissible for $C_4$ are $T_4$ and $[9]$ and in a quotient power graph there exists no proper edge incident to vertices  having the same type, we have that the first proper  edge of $\gamma$  is $\{[\varphi], [\psi]\}$, for some $\psi\in A_9$ such that $T_{\psi}=[9]$. Thus, by  \eqref{lato}, $\{\varphi,\psi\}$ is a proper edge in $P_0(A_9)$ and so necessarily  $\psi^3=\varphi.$ Consider  now the proper edges of $\widetilde{P}_0(A_9)$ of type $\{[\psi], [\delta]\}$, for some $[\delta]\in [A_9]_0$. Again, since the only types admissible for $C_4$ are $T_4$ and $[9]$,
 we get  $T_{\delta}=T_4.$ Thus, by \eqref{lato}, $\psi$ and $\delta$ are one the power of the other. It follows that $\delta=\psi^3=\varphi.$ This means that $\gamma$ is a path of length $1$ in which one end is not of type $T_4.$
Then $\widetilde {c}_0(A_9)_{T_4}=\mu_{T_4}[A_9]=1120.$ Since all the types admissible for $A_9$ have appeared, the procedure \ref{procedureA_n} stops giving $\widetilde{c}_0(A_9)=5442$. The discussion above shows also that $c_0(\mathcal{T}(A_9))=4$ and that $c_0(\mathcal{O}(A_9))=2.$

\end{proof}
\begin{cor} \label{more-even} Let $n\in \mathbb{N}$ with $4\leq n\leq 9.$ Then the vertices of even order in $\widetilde{P}_0(A_n)$ are distributed in more than one component.
\end{cor}

\begin{proof} For $4\leq n\leq 7,$ this is immediate by the case by case analysis of the proof of Theorem \ref {thm:11}.  For $8\leq n\leq 9,$ it follows  by the case by case analysis of the proof of Theorem \ref {thm:15}.
\end{proof}
\vskip 0.4 true cm

\section{ Higher degrees}
\vskip 0.4 true cm
In this section, we determine $c_0(A_n)=\widetilde {c}_0(A_n)$, $c_0(\mathcal{T}(A_n))$ and $c_0(\mathcal{ O}(A_n))$ for $n\geq 10.$
We start showing that the fact that the vertices of even order are distributed in more than one component is a peculiarity of  $4\leq n\leq 9.$

\begin{lem}\label{lem:16}
For every $n\geq 10$, the vertices of even order of $\widetilde{P}_0(A_n)$  are contained in  $\widetilde{\Omega}_n$.
\end{lem}
\begin{proof}
By Lemma~\ref{lem:12}, all the vertices of type $[1^{n-4}, 2^2]$
are contained in  $\widetilde{\Omega}_n$ of $\widetilde{P}_0(A_n)$. Let $[\pi] \in [A_n]_0$ with $o(\pi)=2k$, for a positive integer $k$. Then $o(\pi^k)=2$, so that $\pi^k$ is the product of $s$ cycles of length two, for some  $s\in\mathbb{N}$ even. If $s=2$, then $[\pi^k]\in \widetilde{\Omega}_n$ and so $[\pi]\in \widetilde{\Omega}_n$.
If $s\geq 4$ then $\pi^k=(a~b)(c~d)(e~f)(g~h)\sigma$,  for suitable distinct $a, b, c, d, e, f, g, h\in N$ and $\sigma\in A_n$, with
$\sigma=id$ or the product of $s-4$ cycles of order two. Let $\varphi=(a~c~e~b~d~f)(g~h)\sigma$. Then $\varphi\in A_n$ and $\varphi^3=\pi^k$. Due to $n\geq 10$, there exist distinct elements $i, j \in N \setminus \{a, b, c, d, e, f, g, h\}$ and  we have the following path:
$$[\pi], [\pi^k], [\varphi], [(a~e~d)(c~b~f)], [(a~e~d)(c~b~f)(g~h)(i~j)], [(g~h)(i~j)].$$
Since $[(g~h)(i~j)]\in \widetilde{\Omega}_n$, we also get  $[\pi]\in \widetilde{\Omega}_n$.
\end{proof}
\begin{thm}\label{thm:17} $c_0(A_{10})=\widetilde {c}_0(A_{10})=29345$, $c_0(\mathcal{T}(A_{10}))=3$ and
$c_0(\mathcal{O}(A_{10}))=1.$
\end{thm}
\begin{proof}
By Lemma~\ref{lem:16}, the vertices of even order in $[A_{10}]_0$ are contained in $\widetilde{\Omega}_{10}$. Thus, for $T_1=[1^6,2^2],$ we have $\widetilde {c}_0(A_{10})_{T_1}=1.$ Moreover from the paths:
$$ [(1~2~3)(4~5~6)], [(1~2~3)(4~5~6)(7~8)(9~10)],$$
and
$$ [(1~2~3~4~5~6~7)], [(1~2~3~4~5~6~7)(8~9~10)], [(8~9~10)], [(8~9~10)(1~2~3~4~5)],$$
$$ [(1~2~3~4~5)], [(1~2~3~4~5)(6~7)(8~9)]$$

\noindent we deduce that the types $[1^4, 3^2]$, $[1^3, 7]$, $[3, 7]$, $[1^7, 3]$, $[1^2, 3, 5]$ and $[1^5, 5]$ are  admissible for  $\widetilde{\Omega}_{10}$. Since $[10]\notin  \mathcal {T}(A_{10}),$ the vertices of type $T_2=[5^2]\in \mathcal {T}(A_{10})$ are instead isolated and thus $\widetilde {c}_0(A_{10})_{T_2}=\mu_{[5^2]}[A_{10}]=18144$. The same argument used for $\widetilde{P}_0(A_9)$ shows also that for $T_3=[1, 3^3]$ we have $\widetilde {c}_0(A_{10})_{T_3}=\mu_{[1, 3^3]}[A_{10}]=11200$. Since all the types admissible for $A_{10}$ have appeared, the procedure \ref{procedureA_n} stops giving $\widetilde{c}_0(A_{10})=29345$ and $c_0(\mathcal{T}(A_{10}))=3$. Moreover, since $\widetilde{o}(\widetilde{\Omega}_{10})$ contains all the prime less than $10,$ we deduce that $c_0(\mathcal{O}(A_{10}))=1.$

\end{proof}
\begin{lem}\label{lem:18}
Let $n\geq 11$. Then the vertices of order $3$ in $\widetilde{P}_0(A_n)$ are contained in $\widetilde{\Omega}_n$.
\end{lem}
\begin{proof}
Let $n\geq 11$ and $[\pi] \in [A_n]_0$, with $o(\pi)=3$. Then $\pi$ is the product of $s\geq 1$ cycles of length $3$. To show that $[\pi]\in V_{\widetilde{\Omega}_n}$, we show that whatever $s$ is,  there exists $\varphi\in A_n$ such that  $\varphi^2=\pi$ with $o(\varphi)$ even and apply Lemma~\ref{lem:16}.
 If $s=1$ then, since $n\geq 11$, there exist $a, b, c, d \in N\setminus M_{\pi}$ and we take $\varphi=\pi^2(a~b)(c~d)$. Let $s=2$ and $\pi=(a~b~c)(d~e~f)$.
 Then we take $\varphi=(a~d~b~e~c~f)(g~h)$, where $g, h\in N\setminus M_{\pi}$. Let $s=3$ and $\pi=(a~b~c)(d~e~f)(g~h~i)$. Then we take $\varphi=(a~d~b~e~c~f)(g~i~h)(j~k)$, where $j, k\in N\setminus M_{\pi}$. Finally, let $s\geq 4.$ Then $n\geq 12$ and $\pi=(a~b~c)(d~e~f)(g~h~i)(j~k~l)\sigma$, for suitable $a,b,c,d,e,f,g,h,i,j,k,l\in N$ and $\sigma\in A_n$, with $\sigma^3=id$. We then take $\varphi=(a~d~b~e~c~f)(g~j~h~k~i~l)\sigma^2$.
\end{proof}
\begin{lem}\label{lem:19}
Let $n\geq 11$ and $p$ be a prime number such that $\,5\leq p\leq n-3$ and $ p\neq {\frac{n}{2}, \frac{n-1}{2}}.$  Then the vertices of order $p$ in $\widetilde{P}_0(A_n)$ are contained in $\widetilde{\Omega}_n$.
\end{lem}
\begin{proof}
Let $[\pi] \in [A_n]_0$, with $o(\pi)=p$. Then $T_{\pi}\in\mathcal{T}$, where $$\mathcal{T}=\{T_s=[1^{n-sp},p^s]: s\in\mathbb{N},s\leq n/p\}. $$ By Lemma \ref{lem:12} iii) it is enough to  show that $\mathcal{T}\subseteq \mathcal{T}(\widetilde{\Omega}_n).$ Recall now that, by Lemma \ref{C(T)}, $\mathcal{T}(\widetilde{\Omega}_n)$ is the set of vertices of a component of $P_0(\mathcal{T}(A_n)).$
Thus, by Lemma ~\ref{lem:16} and Lemma~\ref{lem:18}, it is enough to prove that for all $s\in\mathbb{N}$, with $s\leq n/p$, there exist $T^*\in \mathcal{T}(A_n)$ with $o(T^*)$ even or $o(T^*)=3$ and a path in $P_0(\mathcal{T}(A_n))$ between $T^*$ and $T_s$.
If $sp\leq n- 3$, consider $T=[p^s,3,1^{n-sp-3}]$. Since $T^3=T_s$,  we take $T^*=T^p=[3,1^{n-3}]$. If $sp\geq n-2$, since $p\leq n-3$, we have $s\geq 2$. Let first $s=2.$ Then $ n-2\leq 2p\leq n$ and, by $ p\neq {\frac{n}{2}, \frac{n-1}{2}}$, we get $2p=n-2.$
So we can consider $T^*=[2p,2]$, because $o(T^*)$ is even and
$(T^*)^2=T_2$. Next let $s=3$ and consider $T=[3p,1^{n-3p}]$; as $p\geq 5$, we have $T^3=T_3$ while $o(T^p)=3$ and we take $T^*=T^p.$
 Finally,  for $s\geq 4,$  we take $T^*=[(2p)^2,p^{s-4},1^{n-sp}]$ after having observed that $(T^*)^2=T_s$ and that $o(T^*)$ is even.
 \end{proof}
\begin{lem}\label{lem:20}
Let $n\geq 11$ with $ n=2p$ or $n= 2p+1$, for some prime $p$. Let $[\pi]\in [A_n]_0$, with $o(\pi)=p$. Then, the following holds:
 \begin{itemize}\item[(i)] if $T_{\pi}=[1^{n-2p},p^2]$, then $[\pi]$ is an isolated vertex of $\widetilde{P}_0(A_n)$;
 \item[(ii)] if  $T_{\pi}= [1^{n-p},p]$, then  $[\pi]\in \widetilde{\Omega}_n$;
 \item[(iii)] $p\in V_{\widetilde{o}(\widetilde{\Omega}_n)}$.
   \end{itemize}
\end{lem}
\begin{proof}  If $T_{\pi}=[1^{n-2p},p^2]$, since $n-2p\leq 1$, $T_{\pi}$ admits no proper power and is not a proper power of a type admissible for $A_n.$ Thus, by Lemma \ref{isolated},  $[\pi]$ is isolated in $\widetilde{P}_0(A_n).$ Next let $T_{\pi}= [1^{n-p},p]$. By $n\geq 11$, we get $n-p\geq p\geq 5$ and thus, the type $T=[1^{n-p-3},3,p]\in\mathcal{T}(A_n)$ satisfies $o(T^p)=3$, while $T^3=T_{\pi}$. So, by Lemma \ref{lem:18}, $T^p\in \mathcal{T}(\widetilde{\Omega}_n)$ and, by Lemma \ref{C(T)}, $T_{\pi}\in \mathcal{T}(\widetilde{\Omega}_n).$ Finally, Lemma \ref{lem:12} iii), guarantees $[\pi]\in \widetilde{\Omega}_n$. In particular, $p=\widetilde{o}([\pi])\in V_{\widetilde{o}(\widetilde{\Omega}_n)}$.
\end{proof}
\begin{lem}\label{lem:21}
For $n\geq 11$, all the vertices of $\widetilde{P}_0(A_n)$ apart those of order a prime $p$ such that $p\geq n-2$ or $p={\frac{n}{2}, \frac{n-1}{2}}$ are contained in $\widetilde{\Omega}_n$.
\end{lem}
\begin{proof}
Let $[\pi]\in [A_n]_0$, with $o(\pi)$ not a  prime $p$ such that $p\geq n-2$ and not a prime $p$ such that $p={\frac{n}{2}, \frac{n-1}{2}}$. By Lemma \ref{lem:16}, we can assume that $o(\pi)$ is odd. Let first $o(\pi)=p$, for some prime $p.$ If $p=3$, then Lemma \ref{lem:18} applies. If $p\geq 5,$ then,  since $p\leq n-3$, Lemma \ref{lem:19} applies. Assume now that  $o(\pi)$ is composite. If $o(\pi)=3^k$ for some $k\in \mathbb{N}, k\geq 2$, then we have $o(\pi^{3^{k-1}})=3$ and, by Lemma \ref{lem:18}, we get $[\pi^{3^{k-1}}]\in\widetilde{\Omega}_n$, so that also $[\pi]\in \widetilde{\Omega}_n$.
Finally let $o(\pi)=tpq$,  with $p \geq 3, q\geq 5$ prime numbers and $t$ an odd positive integer. Since $o(\pi^t)=pq$, then  $T_{\pi}$ contains a term equal to $pq$ or two terms equal to $p$ and $q$ respectively. In any case, as $n\geq 11$, this implies $o(\pi^{tp})= q\leq n-3$ so that, by Lemma~\ref{lem:19}, $[\pi^{tp}]\in \widetilde{\Omega}_n$ and hence also $[\pi]\in \widetilde{\Omega}_n$.
\end{proof}

\begin{proof}[Proof of Theorem A ] (i) Combine Theorems \ref{thm:11}, \ref{thm:15} and \ref{thm:17}.

(ii) Let $n\geq 11$ be fixed and recall that
 $c_0(A_n)=\widetilde {c}_0(A_n)$. First note that, as $P\cap O_0(A_n)=\{p\in P: p\leq n\},$  we can reformulate Lemma \ref{lem:21} saying that all the vertices of  $\widetilde{P}_0(A_n)$ of order not belonging to \[B(n)=P\cap \Big\{n, n-1, n-2, {\frac{n}{2}, \frac{n-1}{2}}\Big\}\] are in $\widetilde{\Omega}_n.$ We refer to the numbers in $B(n)$ as to the critical orders for $n$.  By Lemma \ref{C(T)} ii), there exists a  unique component  $\Theta_n$ of $\mathcal{O}_0(A_n)$ containing the connected subgraph $\widetilde {o}(\widetilde{\Omega}_n)$, and we have $V_{\Theta_n}\supseteq V_{\widetilde {o}(\widetilde{\Omega}_n)}\supseteq O_0(A_n)\setminus B(n).$

If $n\notin A$, then $B(n)=\varnothing$ and thus all the vertices in $[A_n]_0$ belongs to $\widetilde{\Omega}_n$. So, by \eqref{connected}, $c_0(A_n)=\widetilde {c}_0(A_n)=c_0(\mathcal{T}(A_n))=c_0(\mathcal{O}(A_{n}))=1.$

Next let $n\in A$.
We examine all the possible  positions of $n$ with respect to the sets $P,\, P+1,$ $P+2,\,2P,\,2P+1$ whose union is $A.$ 
First of all  we show that $ P\cap(P+2)\cap (2P+1)=\varnothing$, proving that if $n\in P\cap(P+2)$, then ${\frac{n-1}{2}}\notin P.$
Let $n=p\in P$ and $n-2=p-2=q\in P$. Then, due to $n\geq 11$, we have  $p, q\geq 11$.  Since $3$ divides $n(n-1)(n-2)=p(p-1)q$,  we deduce that $3$ also divides  ${\frac{p-1}{2}=\frac{n-1}{2}}.$ The possibility $3={\frac{n-1}{2}}$ is excluded by $n\geq 11$ and thus ${\frac{n-1}{2}}\notin P.$

Since $2P,\,P+1$ are subsets of the even positive integers, while $P, P+2, 2P+1$ are subsets of the odd positive integers, we have nine cases to deal with. For all those cases we need to understand the number and nature of the components of $\widetilde {P}_0(A_n)$ and $P_0(\mathcal{T}(A_n))$
containing a vertex of order belonging to $B(n)$;  the number and nature of components of  $\mathcal{O}_0(A_n),$ different from  $\Theta_n$, containing a vertex belonging to $B(n)$. Once that is done, by Lemma \ref{component-quotient}, we immediately reach full information for the components of $P_0(A_n)$ too. The counting of $\widetilde {c}_0(A_n)$ is carried on applying Procedure \ref{procedureA_n}.

Let $n\in 2P\setminus (P+1)$. Then $n$ is even, $ {\frac{n}{2}}=p\in P$ and $n-1\notin P$, so that the only critical order is $p.$  But, by Lemma \ref{lem:20}, $p$ is a vertex of $\widetilde {o}(\widetilde{\Omega}_n)$ and thus $V_{\Theta_n}=O_0(A_n)=V_{\widetilde {o}(\widetilde{\Omega}_n)}$  and $c_0(\mathcal{O}(A_{n}))=1.$ To examine the graph  $\widetilde {P}_0(A_n)$,
let $[\pi]\in[A_n]_0$ with $o(\pi)= p.$ Then $T_{\pi}\in \{[p^2], [1^p,p]\}$ and, by Lemma \ref{lem:20} the vertices of type $[1^p,p]$ are in  $\widetilde{\Omega}_n,$ while those of type $T_1=[p^2]$ are isolated.  Thus $\widetilde {c}_0(A_n)_{T_1}=\mu_{[p^2]}[A_{n}]$ and hence
$\widetilde {c}_0(A_n)=\frac{4(n-1)(n-3)!}{n}+1.$ Moreover $c_0(\mathcal{T}(A_n))=2.$

Let $n\in 2P\cap (P+1)$. Then $n$ is even, ${\frac{n}{2}}=p\in P$ and $n-1=q\in P$, so that the critical orders are the primes $p$ and $q.$ By Lemmas \ref{lem:20} and \ref{p-isolated}, we have that $p\in V_{\widetilde {o}(\widetilde{\Omega}_n)}\subseteq V_{\Theta_n}$ while $q$ is isolated in $\mathcal{O}_0(A_{n})$. Thus $c_0(\mathcal{O}(A_{n}))=2$ and, by Lemma \ref{isolated}, $V_{\Theta_n}=O_0(A_n)\setminus \{q\}=V_{\widetilde {o}(\widetilde{\Omega}_n)}.$
To examine the graph  $\widetilde {P}_0(A_n)$, let first
 $[\pi]\in[A_n]_0$  with  $o(\pi)=p$. As in the previous case, $T_{\pi}\in \{T_1=[p^2], T=[1^p,p]\}$ and $T\in \mathcal{T}(\widetilde{\Omega}_n)$ while $\widetilde {c}_0(A_n)_{T_1}=\mu_{[p^2]}[A_{n}].$
 Moreover each $[\pi]\in[A_n]_0$  with
 $o(\pi)=q$ is isolated and $T_{\pi}=[1,q]$. So, by Lemma \ref{isolated}, $T_2=[1,q]$  is isolated in $P_0(\mathcal{T}(A_n))$ and $\widetilde {c}_0(A_n)_{T_2}=n(n-3)!$. It follows that $\widetilde {c}_0(A_n)=\frac{4(n-1)(n-3)!}{n}+n(n-3)!+1$ and $c_0(\mathcal{T}(A_n))=3.$

Let $n\in (P+1)\setminus 2P$. Here  $n$ is even, ${\frac{n}{2}}\notin P$ and $n-1=p\in P$, so that the only critical order is the prime $p$. By Lemma \ref{p-isolated}, $p$ is isolated in $\mathcal{O}_0(A_{n})$ and so $c_0(\mathcal{O}(A_n))=2.$ Thus, by Lemma \ref{isolated}, we also get $V_{\Theta_n}=O_0(A_n)\setminus \{p\}=V_{\widetilde {o}(\widetilde{\Omega}_n)}.$ Moreover, each $[\pi]\in[A_n]_0$ with $o(\pi)= p$ is isolated in $\widetilde {P}_0(A_n).$ Thus  $T_1=[1,p]$ is isolated in $P_0(\mathcal{T}(A_n))$ and
 $\widetilde {c}_0(A_n)_{T_{1}}=n(n-3)!$, so that  $\widetilde {c}_0(A_n)=n(n-3)!+1$ and $c_0(\mathcal{T}(A_n))=2.$

Let $n\in P\setminus [(P+2)\cup (2P+1)]$. Here the only critical order is $n$, which by Lemma \ref{p-isolated} is isolated, so that $c_0(\mathcal{O}(A_n))=2.$ Then, by Lemma \ref{isolated}, $V_{\Theta_n}=O_0(A_n)\setminus \{n\}=V_{\widetilde {o}(\widetilde{\Omega}_n)}.$
For $T_1=[n]$, we get
 $\widetilde {c}_0(A_n)_{T_{1}}=(n-2)!$, so that $\widetilde {c}_0(A_n)=(n-2)!+1$ and $c_0(\mathcal{T}(A_n))=2.$

Let $n\in (2P+1)\setminus [P\cup (P+2)]$. Here the only critical order is the prime $p={\frac{n-1}{2}}$, which by Lemma \ref{lem:20}, is a vertex of $\widetilde {o}(\widetilde{\Omega}_n)$. Thus $c_0(\mathcal{O}(A_n))=1$ and $V_{\Theta_n}=O_0(A_n)=V_{\widetilde {o}(\widetilde{\Omega}_n)}.$ Moreover $T_1=[1,p^2]$ is isolated in $P_0(\mathcal{T}(A_n))$ and $\widetilde {c}_0(A_n)_{T_{1}}=\mu_{T_1}[A_{n}]$. Thus
$\widetilde {c}_0(A_n)={\frac{4n!}{(n-1)^2(n-3)}}+1$ and
$c_0(\mathcal{T}(A_n))=2.$

Let $n\in (P+2)\setminus [P\cup (2P+1)]$. Here the only critical order is the prime $n-2.$ By Lemmas \ref{p-isolated} and \ref{isolated}, $n-2$ is isolated in $\mathcal{O}_0(A_n)$  while the type $[1^2,n-2]$ is isolated in  $P_0(\mathcal{T}(A_n))$. Then $c_0(\mathcal{O}(A_n))=2,$ $V_{\Theta_n}=V_{\widetilde {o}(\widetilde{\Omega}_n)},$
$\widetilde {c}_0(A_n)={\frac{n(n-1)(n-4)!}{2}}+1$ and $c_0(\mathcal{T}(A_n))=2.$

Let $n\in P\cap (P+2)$. Here the critical orders are the primes $n, n-2$, both isolated in $\mathcal{O}_0(A_n)$, by Lemma \ref{p-isolated}. Moreover the types $[n], [1,n-1]$ are isolated in $P_0(\mathcal{T}(A_n))$.
Thus $c_0(\mathcal{O}(A_n))=3,$  $\widetilde {c}_0(A_n)=(n-2)!+{\frac{n(n-1)(n-4)!}{2}}+1$ and $c_0(\mathcal{T}(A_n))=3.$ Applying Lemma \ref{isolated} we also get $V_{\Theta_n}=V_{\widetilde {o}(\widetilde{\Omega}_n)}.$

Let $n\in P\cap (2P+1)$.  The critical orders are the primes $n,{\frac{n-1}{2}}.$ By Lemmas \ref{p-isolated}, \ref{isolated} and \ref{lem:20}, with the usual arguments, we get $c_0(\mathcal{O}(A_n))=2,$  $V_{\Theta_n}=V_{\widetilde {o}(\widetilde{\Omega}_n)},$ $\widetilde {c}_0(A_n)=(n-2)!+\frac{4n(n-2)(n-4)!}{n-1}+1$ and
$c_0(\mathcal{T}(A_n))=3.$ Moreover the types $[n]$ and $[1,(\frac{n-1}{2})^2]$ are isolated in $P_0(\mathcal{T}(A_n))$.

Let $n\in (P+2)\cap (2P+1)$.  The critical orders are the primes $n-2$ and $\frac{n-1}{2}$. By Lemmas \ref{p-isolated} and \ref{lem:20}, $n$ is isolated in $\mathcal{O}_0(A_n)$ while $\frac{n-1}{2}\in V_{\widetilde {o}(\widetilde{\Omega}_n)}.$ Thus $c_0(\mathcal{O}(A_n))=2$ and $V_{\Theta_n}=V_{\widetilde {o}(\widetilde{\Omega}_n)}.$ Moreover the types $[1^2,n-2]$ and $[1,(\frac{n-1}{2})^2]$ are isolated in $P_0(\mathcal{T}(A_n))$, so that
  $\widetilde {c}_0(A_n)=\frac{n(n-1)(n-4)!}{2}+\frac{4n(n-2)(n-4)!}{n-1}+1$ and $c_0(\mathcal{T}(A_n))=3.$
  \end{proof}

\begin{proof}[Proof of Theorem B ]
The case by case analysis in the proof of  Theorem A has shown many facts. In particular, $\widetilde{\Omega}_n$ is the only possible component of $\widetilde{P}_0(S_n)$ not reduced to an isolated vertex and $V_{\Theta_n}=V_{\widetilde {o}(\widetilde{\Omega}_n)}.$ Since $\widetilde {o}$ is complete, by \cite[Proposition 5.2]{bub}, we then deduce $\Theta_n=\widetilde {o}(\widetilde{\Omega}_n).$

Call now main component of $P_0(A_n)$,  $\widetilde{P}_0(A_n)$, $P_0(\mathcal{T}(A_n))$ and $\mathcal{O}_0(A_n)$, respectively, the component $\Omega_n$  such that $\pi(\Omega_n)=\widetilde{\Omega}_n$ defined in Corollary \ref{complete}, $\widetilde{\Omega}_n$,  $\widetilde{t}(\widetilde{\Omega}_n)$ and
$\widetilde {o}(\widetilde{\Omega}_n)$. Note that $V_{\widetilde{t}(\widetilde{\Omega}_n)}=\mathcal{T}(\widetilde{\Omega}_n)$ and that $\widetilde{t}(\widetilde{\Omega}_n)$ is a component  because $\widetilde{t}$ is a pseudo-covering.
By Corollary \ref{complete}, no main component is complete.

We show  now that, for every $X_0\in \{\widetilde{P}_0(A_n),\ P_0(\mathcal{T}(A_n))\}$, all the components except the main are isolated vertices of order a prime $p\in B(n).$ Let $\widetilde{C}\in  \widetilde{\mathcal{C}}_0 (A_n)$, with $\widetilde{C}\neq \widetilde{\Omega}_n$. Choose $[\psi]\in V_{\widetilde{C}}$, so that $\widetilde{C}=C([\psi]).$ Since $\widetilde{t}$ is pseudo-covering, by Proposition \ref{hom-con}\,(ii)\,(b), we get $\widetilde{t}(\widetilde{C})=\widetilde{t}(C([\psi]))=C(T_{\psi})$. But, from $[\psi]\notin V_{\widetilde{\Omega}_n}$ we deduce, by Lemma \ref{lem:21}, that $o(\psi)\in B(n)$ and thus, by our case by case  analysis, $T_{\psi}$ is isolated. By Lemma \ref{isolated}\,(i), this is equivalent to $[\psi]$ isolated and so $\widetilde{C}$ is reduced to the vertex $[\psi]$.
Next let $C\in \mathcal{C}_0(\mathcal{T}(A_n))$, with $C\neq \widetilde{t}(\widetilde{\Omega}_n)$. By Proposition \ref{hom-con}\,(ii)\,(b), we have $C=\widetilde{t}(\widetilde{C})$ for a suitable $\widetilde{C}\in  \widetilde{\mathcal{C}}_0 (A_n)$ and $\widetilde{C}\neq \widetilde{\Omega}_n$ and, by the above case, we know that $\widetilde{C}$ is an isolated vertex $[\psi]$ of $\widetilde{P}_0(A_n)$. Thus, by Lemma \ref{isolated}\,(i),  $T_{\psi}\in V_C$ is an isolated vertex of $P_0(\mathcal{T}(A_n))$ and so $C$ is reduced to the vertex $T_{\psi}.$

Now, to obtain that each component of $P_0(A_n)$ different from $\Omega_n$
is a complete graph on $p-1$ vertices it is enough to invoke Lemma \ref{component-quotient}.
Finally observe that the case by case analysis in the proof of  Theorem A has shown that $|B(n)|\leq 2,$ which says that the primes involved in components different from the main one are at most $2$ for all the graphs in $\mathcal{G}_0.$

\end{proof}

 \begin{cor} \label{final} The minimum $n\geq 4$ such that $P(A_n)$ is $2$-connected  is $n=16$. There exists infinitely many $n$ such that $P(A_n)$ is $2$-connected.
 \end{cor}
 \begin{proof}  Let first $4\leq n\leq 15$. Then we have $4\leq n\leq 10$,  or $n\geq 11$ and $n\in A$. Thus, by Theorem A, we get $c_0(A_{n})>1$. Moreover, since $16\in \mathbb{N}\setminus A$, Theorem A implies that $c_0(A_{16})=1$. To prove that $\mathbb{N}\setminus A$ is infinite we show that it contains $ B=\{ 4k^2: k\in  \mathbb{N}, k \geq 2 \}.$
 Let $n = 4k^2\in B$. Since $n$ is even and $n\geq 16$, we have $n\notin P\cup (P+2)\cup (2P+1)$. Moreover ${\frac{n}{2}}=2k^2$ is not a prime so that $n\notin 2P.$ Finally, we have $n-1=(2k -1)(2k + 1)$, so that $n\notin P+1.$ Hence $n\in \mathbb{N}\setminus A$.
 \end{proof}
 Note that there exists also some $n$ odd  such that $P(A_n)$ is $2$-connected. For instance this happens for $n=51.$
 \begin{proof} [Proof of Corollary C]  It is immediate by checking that $c_0(\mathcal{T}(A_n))=1$ if and only if  $c_0(A_n)=1$, using Table $1$ and Table $2$,  and noting that the natural numbers $n$ such that $4\leq n\leq 10$ are contained in $A.$ To argue about the infinitely many $n\in \mathbb{N}$ such that $P(A_n)$ is $2$-connected, we use Corollary \ref{final}.
 \end{proof}
 
We finally give a more concise overview on the number of components of the graph $\mathcal{O}_0(A_n).$
\begin{cor}\label{thm:23} The values of $c_0(\mathcal{O}(A_n))$ for $n\geq 4$, with $ n\neq 6,$ are given in Table \ref{eqtable6} below.

\begin{table}[ht]
\caption{$c_0(\mathcal{O}(A_n))$, for $n\geq 4,\  n\neq 6$.}\label{eqtable6}
\begin{center}
\begin{tabular}{|c|c|}

\hline
\vspace{-3mm}
& \\

$c_0(\mathcal{O}(A_n))$  & $n\geq 4,\  n\neq 6$\\
\vspace{-3mm}
& \\

\hline
\vspace{-3mm}
& \\

$1$ & $n\notin P\cup (P+1)\cup (P+2)$\\
\vspace{-3mm}
& \\

\hline
\vspace{-3mm}
& \\

$2$ & $n\in [P\setminus (P+2)]\cup [(P+2)\setminus P]\cup (P+1)$\\
\vspace{-3mm}
& \\

\hline
\vspace{-3mm}
& \\

$3$ & $n\in P\cap (P+2)$\\

\hline
\end{tabular}
\end{center}
\end{table}
\end{cor}
\begin{proof}  A check on the Tables $1$ and $2$.

\end{proof}

 \begin{proof} [Proof of Corollary D]  It follows from Theorem \ref{thm:11} and Corollary \ref{thm:23}.
 \end{proof}

 \vskip 0.4 true cm
 {\bf Acknowledgements.}  The first author is partially supported by GNSAGA of INdAM.
\vskip 0.4 true cm


\bigskip
\bigskip

{\footnotesize \pn{\bf Daniela~Bubboloni}\; \\  {Dipartimento di Matematica e Informatica U.Dini},\\
{ Viale Morgagni 67/a},{ 50134 Firenze, Italy}\\
{\tt Email:  daniela.bubboloni@unifi.it}\\

{\footnotesize \pn{\bf Mohammad~A.~Iranmanesh}\; \\ {Department of
Mathematics}, {Yazd University,\\  89195-741,} { Yazd, Iran}\\
{\tt Email: iranmanesh@yazd.ac.ir}\\

{\footnotesize \pn{\bf Seyed~M.~Shaker}\; \\  {Department of
Mathematics}, {Yazd University,\\  89195-741,} { Yazd, Iran}\\
{\tt Email: seyed$_{-}$shaker@yahoo.com}\\

\end{document}